\documentclass[11pt,a4paper]{article}%
\usepackage{amssymb,amsmath,amsfonts,amsthm,array,bm,color}%

\usepackage{graphicx}
\providecommand{\U}[1]{\protect\rule{.1in}{.1in}}
\numberwithin{equation}{section}

\newtheorem{theorem}{Theorem}[section]
\newtheorem{lemma}[theorem]{Lemma}
\newtheorem{corollary}[theorem]{Corollary}
\newtheorem{proposition}[theorem]{Proposition}
\newtheorem{remark}[theorem]{Remark}

\newtheorem{hypothesis}[theorem]{Hypothesis}

\def\R{{\mathbb R}}
\def\E{{\mathbb E}}
\def\P{{\mathbb P}}
\def\Q{{\mathbb Q}}
\def\N{{\mathbb N}}
\def\L{{\mathcal L}}
\def\eps{\varepsilon}
\def\d{{\rm d}}
\def\<{\langle}
\def\>{\rangle}

\begin{document}

\title{On the relation between the Girsanov transform and the Kolmogorov equations for SPDEs}

\author{Franco Flandoli\footnote{Email: franco.flandoli@sns.it. Scuola Normale Superiore, Piazza dei Cavalieri, 7, 56126 Pisa, Italy.}
\quad Dejun Luo\footnote{Email: luodj@amss.ac.cn. Key Laboratory of RCSDS, Academy of Mathematics and Systems Science, Chinese Academy of Sciences, Beijing 100190, China, and School of Mathematical Sciences, University of the Chinese Academy of Sciences, Beijing 100049, China.}
\quad Cristiano Ricci\footnote{Email: cristiano.ricci@sns.it. Scuola Normale Superiore, Piazza dei Cavalieri, 7, 56126 Pisa, Italy.}}
\maketitle

\vskip-20pt

\begin{abstract}
The Girsanov transform and Kolmogorov equations are two useful methods for studying SPDEs. It is shown that, under suitable conditions, the series expansion obtained from the Girsanov transform coincides with the one generated by an iteration scheme for Kolmogorov equations. We also apply the iteration approach to extend the well posedness theory for Kolmogorov equations beyond the boundedness condition on the nonlinear term.
\end{abstract}


\section{Introduction}

Consider the stochastic equation%
\begin{equation}\label{spde}
\left\{ \aligned
\d X_{t}  & =\left(  AX_{t}+B\left(  X_{t}\right)  \right)  \d t+\sqrt{Q}\,\d W_{t},\\
X_{0}  & =x
\endaligned \right.
\end{equation}
in a Hilbert space $H$, with the norm and inner product $|\cdot|$ and $\<\cdot, \cdot\>$. Here $A$ is the infinitesimal generator of a strongly continuous semigroup $\{e^{tA} \}_{t\geq 0}$ in $H$, $Q$ is a nonnegative self-adjoint operator in $H$ satisfying ${\rm Ker}(Q)=\{0\}$, and $B:H\to H$ is a measurable mapping. Finally, $\{W_t \}_{t\geq 0}$ is a cylindrical Wiener process on $H$, defined on some probability space $(\Omega, \mathcal F, \mathcal F_t, \P)$. The infinite dimensional Kolmogorov equation corresponding to \eqref{spde} is
  \begin{equation}\label{KolE}
  \left\{ \aligned
  \partial_t u(t,x)&= \frac12 {\rm Tr}(Q D^2u(t,x)) + \<Ax+ B(x), Du(t,x)\>, \\
  u(0,x)&= \phi(x),
  \endaligned \right.
  \end{equation}
where $D$ is the spatial derivative operator and $\phi:H\to\R$ is a measurable function.

The underlying linear equation reads as
\begin{equation}\label{linear-eq}
\left\{ \aligned
\d Z^x_{t}  & = AZ^x_{t}\,  \d t+\sqrt{Q}\,\d W_{t},\\
Z^x_{0}  & =x.
\endaligned \right.
\end{equation}
Define the operators
  \begin{equation}\label{covar-oper}
  Q_t= \int_0^t e^{sA} Q e^{sA^\ast} \,\d s,\quad t>0,
  \end{equation}
where $A^\ast$ is the adjoint operator of $A$. If ${\rm Tr}(Q_t)<\infty$ for all $t>0$, then the equation \eqref{linear-eq} has the mild solution
  $$Z^x_t= e^{tA}x+ W_A(t),$$
where $W_A(t)$ is the stochastic convolution:
  $$W_A(t)= \int_0^t e^{(t-s)A} \sqrt{Q}\,\d W_s. $$
The process $\{Z^x_t \}_{t\geq 0}$ is usually called the Ornstein-Uhlenbeck process in the literature; for any $t>0$, $Z^x_t$ has the Gaussian law $N_{e^{tA}x, Q_t}$ with mean $e^{tA}x\in H$ and covariance operator $Q_t$.

Let $\{S_t\}_{t\in [0,T]}$ be the Ornstein-Uhlenbeck semigroup associated to \eqref{linear-eq}. Then we can rewrite the Kolmogorov equation \eqref{KolE} in the mild form:
  \begin{equation}\label{mild-sol}
  u(t,x)= S_t\phi(x) + \int_0^t S_{t-s}\big(\<B, Du(s)\> \big)(x)\,\d s.
  \end{equation}
One can solve this equation in suitable spaces by using the contraction mapping principle, see e.g. \cite[Section 9.4.2]{DPZ}. In the recent paper \cite{FLR}, we have exploited this idea and studied the iterative approximation for the solution: $u_0(t,x)= S_t\phi(x)$ and
  $$u_n(t,x)= S_t\phi(x) + \int_0^t S_{t-s}\big(\<B, Du_{n-1}(s)\> \big)(x)\,\d s, \quad n\geq 1.$$
Define the functions: $v_0(t,x)= u_0(t,x)= S_t\phi(x)$ and
  $$v_n(t,x)= u_n(t,x)- u_{n-1}(t,x),\quad n\geq 1.$$
Then, we obtain the iteration scheme below:
  \begin{equation}\label{new-iteration}
  \left\{ \aligned
  v_{n+1}(t,x)&= \int_0^t \big(S_{t-s} k^n_s\big) (x)\,\d s, \\
  k^n_s(y)&= \<B(y), Dv_n(s,y)\>,\\
  v_0(t,x)&= S_t\phi(x).
  \endaligned \right.
  \end{equation}
The main result of \cite{FLR} (see Theorem 1.1 therein) can be stated as follows.

\begin{proposition}\label{1-prop}
Assume that $\phi:H\to \R$ and $B:H\to H$ are bounded; then under suitable conditions on the operators $A$ and $Q$, the following series
  \begin{equation}\label{1-prop.1}
  u(t,x)= \sum_{n=0}^\infty v_n(t,x)
  \end{equation}
converge uniformly in $(t,x)\in [0,T]\times H$, where, for $n\geq 1$,
  $$ \aligned
  v_n ( t,x )  =& \int_{0}^{t} \d r_n \int_0^{r_n} \d r_{n-1} \cdots \int_{0}^{r_2} \d r_1 \\
  &\, \mathbb{E}\Bigg[ \phi\big( Z^x_{t} \big)
  \prod_{i=1}^{n} \Big\langle \Lambda ( r_{i+1}-r_{i} )  B\big( Z^x_{r_i} \big) ,Q_{r_{i+1}-r_{i}}^{-1/2} \big(Z^x_{r_{i+1}}-e^{(  r_{i+1}-r_{i})  A}Z^x_{r_i} \big) \Big\rangle \Bigg] ,
  \endaligned $$
where $r_{n+1}=t$ and $\Lambda(t)= Q_t^{-1/2} e^{tA},\, t>0$.
\end{proposition}

The above series expansion gives us an explicit formula of $u(t,x)$ in terms of Gaussian integrals, which leads to an algorithm for numerical solution of \eqref{KolE}, as demonstrated by various examples in \cite{FLR}.

We rewrite the formula \eqref{1-prop.1} as
  \begin{equation}\label{eq-1}
  u(t,x)= \E\big[ \phi(Z^x_t) \rho(t,x)\big],
  \end{equation}
in which
  $$\aligned
  \rho(t,x)=  1+ \sum_{n=1}^\infty & \int_{0}^{t} \d r_n \int_0^{r_n} \d r_{n-1} \cdots \int_{0}^{r_2} \d r_1 \\
  &\, \prod_{i=1}^{n} \left\langle \Lambda ( r_{i+1}-r_{i} )  B\big( Z^x_{r_i} \big) ,Q_{r_{i+1}-r_{i}}^{-1/2} \big(Z^x_{r_{i+1}}-e^{(  r_{i+1}-r_{i})  A}Z^x_{r_i} \big) \right\rangle.
  \endaligned $$
The formula \eqref{eq-1} looks very close to the one obtained from the Girsanov transform; indeed, the Girsanov transform also yields a series expansion of the form \eqref{1-prop.1}, see e.g. Proposition \ref{prop-Girsanov} below. The purpose of this paper is to rigorously establish this relation. In the following we write $\mathcal L(H)$ for the Banach space of bounded linear operators on $H$ with the norm $\|\cdot\|_{\mathcal L(H)}$.

\begin{hypothesis}\label{hypothe}
\begin{itemize}
\item[\rm (i)] $A:D(A)\subset H\to H$ is the infinitesimal generator of a strongly continuous semigroup $\{ e^{tA}\}_{t\geq 0}$.
\item[\rm (ii)] $Q\in \mathcal L(H)$ is a nonnegative self-adjoint operator satisfying ${\rm Ker}(Q)= \{0\}$, and for any $t>0$, the linear operator defined in \eqref{covar-oper} is of trace class.
\item[\rm (iii)] For any $t>0$ we have $e^{tA}(H) \subset Q_t^{1/2}(H)$; then by the closed graph theorem, $\Lambda(t) = Q_t^{-1/2} e^{tA}$ is well defined as a bounded linear operator.
\item[\rm (iv)] We assume that
  $$\int_0^t \|\Lambda(s) \|_{\mathcal L(H)}\, \d s<\infty , \quad t>0 .$$
\item[\rm (v)] The initial condition $\phi:H\to \R$ and the nonlinear mapping $B:H\to H$ are bounded and uniformly continuous; moreover, $B(H) \subset Q^{1/2}(H)$ and the mapping $Q^{-1/2} B: H\to H$ has at most linear growth.
\end{itemize}
\end{hypothesis}

The assumptions (i)--(iv) are classical in the literature. The first part of (v) will be useful in taking some limits in the proofs below, while the second part is needed in the Girsanov transform (see Remark \ref{1-rem} below). The main result of the paper is

\begin{theorem}\label{main-thm}
Under Hypothesis \ref{hypothe} above, the series expansion obtained from the Girsanov transform coincides with the one in \eqref{1-prop.1}.
\end{theorem}

This result provides us with a link between the Girsanov transform and Kolmogorov equations for SPDEs; it will be proved in Section \ref{subsec-proof}, following some ideas in  \cite[Theorem 3.1]{KY} where the finite dimensional case was treated. Note that the Girsanov transform does not require the nonlinear term $B$ to be bounded, while the usual results on Kolmogorov equations assume boundedness of $B$, see for instance \cite{DPF, DPFPR} and \cite[Section 2]{DPFPR-15}. The nonlinear part considered in \cite{DPFRV} is the sum of a bounded mapping $B$ and a special unbounded term of gradient type $\nabla V$, satisfying some complicated conditions; see also \cite{RS} for some related results.

Our purpose in the rest of the paper is, applying the iterative scheme \eqref{new-iteration}, to extend the theory on Kolmogorov equations beyond the boundedness assumption on $B$. To state our next result, we need the following conditions which are replacements of (ii), (iv) and (v) in Hypothesis \ref{hypothe}.
\begin{hypothesis}\label{hypothe-1}
\begin{itemize}
\item[\rm (ii$'$)] The operator $Q_\infty= \int_0^\infty e^{sA} Q e^{sA^\ast} \,\d s$ is well defined and of trace class; we denote by
  \begin{equation}\label{invariant-measure}
  \mu= N_{ Q_\infty}
  \end{equation}
the centered Gaussian measure on $H$ with covariance operator $Q_\infty$.
\item[\rm (iv$'$)] We assume there exists $\delta\in (0,1)$ and $C_\delta>0$ such that
  $$\|\Lambda(t) \|_{\mathcal L(H)} \leq C_\delta /t^\delta, \quad t>0 .$$
\item[\rm (v$'$)] The initial datum $\phi \in L^{p_0} (H,\mu)$ for some $p_0>1$, and the nonlinear part $B:H\to H$ in \eqref{KolE} has sublinear growth: there exist $C>0$ and $\beta\in (0,2(1-\delta))$ such that
  $$|B(x)| \leq C\big(1+|x|^\beta \big) \quad \mbox{for all } x\in H.$$
\end{itemize}
\end{hypothesis}

Recall that the Gaussian measure $\mu$ is the unique invariant measure of the Ornstein--Uhlenbeck semigroup $\{ S_t\}_{t>0}$, cf. \cite[Theorem 2.34]{DaPrato}. Next, the parameter $\delta$ belongs to $[1/2, 1)$ in many examples (cf. \cite[Lemma 2.3]{Flandoli} or \cite[Example 2.2]{FLR}), thus the speed of growth of $B$ is strictly lower than $|x|$. Our next main result is

\begin{theorem}\label{main-thm-0}
Assume the conditions (i) and (iii) of Hypothesis \ref{hypothe} and Hypothesis \ref{hypothe-1}. Then, for any $\bar p\in (1,p_0)$ and $ T>0$, the Kolmogorov equation \eqref{KolE} has a unique mild solution $u\in C\big([0,T], L^{\bar p}(H, \mu) \big)$ such that $ t^\delta Du(t) \in C\big([0,T], L^{\bar p}(H, \mu;H) \big)$.
\end{theorem}

This result will be proved in Section \ref{sec-Kolmogorov}. It shows that the solution $u$ belongs to the space $ C\big([0,T], L^{p_0-}(H, \mu) \big)$, where $L^{p_0-}(H, \mu)= \cap_{p<p_0} L^{p}(H, \mu)$. By making use of the expression of $v_n(t,x)$ in \eqref{new-iteration} and the H\"older inequality in a clever way, we shall prove some estimates on $v_n$ in spaces $L^{p_n}(H,\mu)$, with a carefully chosen decreasing sequence $\{p_n \}_{n\geq 1}$ such that $p_n\in(\bar p, p_0)$ for all $n\geq 1$. In this way, we can show that the series is convergent in  $L^{\bar p}(H,\mu)$; see the beginning of Section \ref{sec-Kolmogorov} for a detailed explanation of the idea of proof. Finally, we collect in the appendix some moment estimates of Gaussian measures on Hilbert space, which play an important role in Section \ref{sec-Kolmogorov}.

\section{Girsanov transform and proof of Theorem \ref{main-thm}}

This section consists of two parts. In Section \ref{subsec-Girsanov}, we give a brief introduction of the classical Girsanov transform for Eq. \eqref{spde} under suitable conditions; accordingly,  we get a weak solution of \eqref{spde} and the corresponding semigroup, as well as a series expansion for the latter. We then prove the first main result (Theorem \ref{main-thm}) in Section \ref{subsec-proof}.

\subsection{Girsanov transform and the corresponding series expansion}\label{subsec-Girsanov}

By Hypothesis \ref{hypothe}-(v), the following quantity makes sense:
  $$\psi(t,x)= Q^{-1/2} B(Z^x_t), \quad t>0,\, x\in H.$$
Assume that
  $$\P\bigg( \int_0^T |\psi(t,x)|^2\,\d t<+\infty \bigg)=1;$$
then both
  \begin{equation}\label{martingale.1}
  L_t= L^x_t= \int_0^t \<\psi(s,x),\d W_s\>, \quad t\in [0,T]
  \end{equation}
and
  \begin{equation}\label{martingale.2}
  M_t=M^x_t= \exp\bigg(\int_0^t \<\psi(s,x),\d W_s\> - \frac12\int_0^t |\psi(s,x)|^2\,\d s \bigg), \quad t\in [0,T]
  \end{equation}
are local martingales on the probability space $(\Omega, \mathcal F, \mathcal F_t, \P)$. The well known Girsanov theorem can be stated as below (cf. \cite[Section 10.2]{DPZ} or \cite[Section 10.3]{DaPrato-1}).

\begin{theorem}\label{thm-Girsanov}
Assume that $\E M_t=1$ for all $t\in [0,T]$. Then, under the probability measure
  $$\d\Q= M_T\,\d \P,$$
the process
  $$\tilde W_t= W_t- \int_0^t \psi(s,x)\,\d s,\quad t\in [0,T]$$
is a cylindrical Wiener process on $H$. As a consequence, the process $\{Z^x_t\}_{t\in [0,T]}$ is a (weak) mild solution to \eqref{spde} on the probability space $(\Omega, \mathcal F, \Q)$ with the cylindrical Brownian motion.
\end{theorem}

\begin{remark}\label{1-rem}
By \cite[Proposition 10.17]{DPZ}, a sufficient condition for $\E M_T=1$ is that there exists a $\delta>0$ such that
  \begin{equation}\label{1-rem.1}
  \sup_{t\in [0,T]}\E_\P\, e^{\delta|\psi(t,x)|^2} <\infty.
  \end{equation}
Note that $\psi(t,x)= Q^{-1/2} B(Z^x_t)$ and $Z^x_t$ has the Gaussian law $N_{e^{tA}x, Q_t}$; by Fernique's theorem, the condition \eqref{1-rem.1} follows from the linear growth property of $Q^{-1/2}B$ in (v) of Hypothesis \ref{hypothe} (see  \cite[Theorem 10.20]{DPZ}).
\end{remark}

Since we have a (weak) mild solution to \eqref{spde}, we can represent the associated semigroup: for any $\phi\in \mathcal B(H)$,
  $$P_t\phi(x)= \E_\Q \phi(Z^x_t)= \E_\P\big[\phi(Z^x_t) M_T\big],\quad (t,x)\in [0,T]\times H.$$
By the martingale property,
  \begin{equation}\label{semigroup}
  P_t\phi(x)= \E_\P \Big\{\E_\P\big[\phi(Z^x_t) M_T| \mathcal F_t\big] \Big\}= \E_\P\big[\phi(Z^x_t)\, \E_\P(M_T| \mathcal F_t) \big]= \E_\P\big[\phi(Z^x_t) M_t \big].
  \end{equation}
In the sequel we simply write $\E$ instead of $\E_\P$. Remark that \eqref{1-rem.1} implies the process $\{L_t\}_{t\in [0,T]}$ is a martingale having finite moments of all orders; indeed, for any $p\geq 1$, by the Burkholder-Davis-Gundy inequality,
  \begin{equation}\label{marting-integrability}
  \aligned
  \E\bigg[\sup_{t\leq T} L_t^{2p}\bigg]&\leq C_p \E\bigg[\Big(\int_0^T |\psi(t,x)|^2\,\d t \Big)^p\bigg] \leq C_p T^{p-1} \int_0^T \E\big( |\psi(t,x)|^{2p} \big) \,\d t \\
  &\leq C_p T^p \sup_{t\leq T} \E\big( |\psi(t,x)|^{2p} \big) <+\infty,
  \endaligned
  \end{equation}
where the last step follows from \eqref{1-rem.1}. It is clear that the martingale $M_t$ satisfies the stochastic equation $\d M_t =M_t \<\psi(t,x),\d W_t\> = M_t\, \d L_t$, therefore
  \begin{equation}\label{martingale.3}
  M_t= 1+ \int_0^t M_s\,\d L_s.
  \end{equation}
We substitute this formula into the right hand side of \eqref{semigroup}:
  $$P_t\phi(x)= \E\big[\phi(Z^x_t) \big] + \E\bigg[\phi(Z^x_t) \int_0^t M_s\,\d L_s \bigg].$$
Repeating this procedure yields
  $$\aligned
  P_t\phi(x) &= \E\big[\phi(Z^x_t) \big] + \E\big[\phi(Z^x_t) L_t \big] + \E\bigg[\phi(Z^x_t) \int_0^t \bigg( \int_0^s M_r\,\d L_r \bigg)\,\d L_s \bigg].
  \endaligned $$
We can proceed as above to get a series expansion of $P_t\phi(x)$. To simplify the notation we introduce $M^{(0)}_t\equiv 1$ and
  $$M^{(n)}_t= \int_0^t M^{(n-1)}_s\,\d L_s,\quad t\in [0,T], \, n\geq 1. $$
Under the condition \eqref{1-rem.1}, similar to the computation in \eqref{marting-integrability}, we can prove inductively that the martingales  $M^{(n)}_t \ (n\geq 1)$ have finite moments of all orders. Moreover, similarly as in \cite[Proposition 3.1]{KY}, we have

\begin{proposition}\label{prop-Girsanov}
Let $T>0$ and $\phi\in \mathcal B(H)$. Assume that
  $$\big\| Q^{-1/2} B \big\|_\infty = \sup_{x\in H} \big| Q^{-1/2} B(x) \big|<+\infty,$$
then the series
  \begin{equation}\label{prop-Girsanov.1}
  P_t\phi(x)= \sum_{n=0}^\infty I_n(t,x) := \sum_{n=0}^\infty \E\Big[\phi(Z^x_t) M^{(n)}_t\Big], \quad t\in [0,T],\, x\in H
  \end{equation}
converge uniformly on $[0,T] \times H$.
\end{proposition}

\begin{proof}
Under the above assumption, it is easy to show that, by induction,
  $$\E\bigg[\sup_{t\in [0,T]} \Big(M^{(n)}_t \Big)^2 \bigg] \leq C^n \big\| Q^{-1/2} B \big\|_\infty^{2n}\, \frac{t^n}{n!}, $$
where the constant $C$ comes from Doob's maximal inequality. This immediately gives us the desired result.
\end{proof}

\subsection{Proof of Theorem \ref{main-thm}} \label{subsec-proof}

Under Hypothesis \ref{hypothe}, we will show that the terms $I_n(t,x)$ obtained in Proposition \ref{prop-Girsanov} coincide with $v_n(t,x)$ defined in the introduction. First, since $M^{(0)}_t=1$, one has $I_0(t,x)= \E [\phi(Z^x_t) ] = v_0(t,x)$. The following result shows the fact $I_1(t,x)=v_1(t,x)$, for which we present a detailed proof to illustrate the idea. We fix $x\in H$ and write $Z_t= Z^0_t =W_A(t)$ for the stochastic convolution. It holds that
  \begin{equation}\label{subsec-proof.1}
  Z_t- e^{(t-s)A} Z_s =Z^x_t- e^{(t-s)A} Z^x_s, \quad s\in [0,t].
  \end{equation}

\begin{proposition}\label{lem-1}
For any $t\in [0,T]$,
  $$ I_1(t,x)= \int_0^t \E\Big[\phi(Z^x_t)\big\< \Lambda(t-s) B(Z^x_s), Q_{t-s}^{-1/2} (Z_t- e^{(t-s)A} Z_s)\big\> \Big]\, \d s. $$
\end{proposition}

\begin{proof}
We fix $t\in (0,T]$ and consider the backward Kolmogorov equation:
  $$\left\{ \aligned
  & \partial_s U_1(s,y) + \<Ay, D U_1(s,y)\>+ \frac12\text{Tr} \big(Q D^2 U_1(s,y) \big)=0, \quad s\in [0,t],\\
  & U_1(t,y)= \phi(e^{tA}x + y).
  \endaligned \right. $$
Then we have
  \begin{equation}\label{lem-1.1}
  U_1(s,y)= \E\big[\phi(e^{tA}x + Z_t)| Z_s=y\big]= \E\big[\phi\big(e^{tA}x + e^{(t-s)A}y + Z_{t-s} \big)\big] = \E\big[\phi\big(e^{tA}x + Z^y_{t-s} \big)\big].
  \end{equation}
In particular, $U_1(t,Z_t)= \phi(e^{tA}x + Z_t)= \phi(Z^x_t)$. For any $h\in H$, it is well known that  (see e.g. \cite[Proposition 2.28]{DaPrato})
  \begin{equation}\label{lem-1.2}
  \<D U_1(s,y),h\>= \E\Big[\phi(e^{tA}x + Z^y_{t-s})\, \big\<\Lambda(t-s)h, Q_{t-s}^{-1/2} (Z^y_{t-s}- e^{(t-s)A} y) \big\> \Big].
  \end{equation}

Now we are ready to find the expression of $I_1(t,x)$ defined in \eqref{prop-Girsanov.1}. For any $t_0\in (0,t)$, by the It\^o formula and the equation satisfied by $U_1(s,y)$,
  $$U_1(t_0, Z_{t_0})= U_1(0, 0)+ \int_0^{t_0} \big\<D U_1(s, Z_s), \sqrt{Q}\, \d W_s\big\>.$$
Here, we stress that we cannot take $t_0=t$ since, by the following rough estimate, the stochastic integral may not make sense. Indeed, from \eqref{lem-1.2} we conclude that $|D U_1(s,y)| \leq \|\Lambda(t-s)\|_{\L(H)} \|\phi \|_\infty$, thus,
  $$\aligned
  \E \int_0^t \big|\sqrt{Q} D U_1(s, Z_s)\big|^2\,\d s &\leq (\text{Tr} \,Q)\E \int_0^t |D U_1(s, Z_s)|^2\, \d s \\
  &\leq (\text{Tr} \,Q) \|\phi \|_\infty^2 \int_0^t \|\Lambda(t-s)\|_{\L(H)}^2\, \d s .
  \endaligned $$
The last integral is infinite since $\|\Lambda(s)\|_{\L(H)} \geq C/ s^{1/2}$ in examples.

Recall that $\phi(Z^x_t)= U_1(t,Z_t)$; we have
  \begin{equation}\label{lem-1.2.5}
  I_1(t,x)= \E\Big[U_1(t,Z_t) M^{(1)}_t \Big] = \E\Big[\big( U_1(t,Z_t)- U_1(t_0, Z_{t_0})\big) M^{(1)}_t \Big]+ \E\Big[ U_1(t_0, Z_{t_0}) M^{(1)}_t \Big].
  \end{equation}
As remarked above Proposition \ref{prop-Girsanov}, $M^{(1)}_t$ has finite moments of all orders; by Lemma \ref{2-lem-1} below, we obtain
  \begin{equation}\label{lem-1.2.8}
  \lim_{t_0\uparrow t} \E\Big[\big( U_1(t,Z_t)- U_1(t_0, Z_{t_0})\big) M^{(1)}_t \Big] =0.
  \end{equation}
Next, as $M^{(1)}_t= \int_0^t M^{(0)}_s\, \d L_s= L_t= \int_0^t \big\<Q^{-1/2} B(Z^x_s), \d W_s \big\> $ is a martingale, we have
  \begin{equation}\label{lem-1.3}
  \aligned
  \E\Big[ U_1(t_0, Z_{t_0}) M^{(1)}_t \Big]&= \E\Big[ U_1(t_0, Z_{t_0}) M^{(1)}_{t_0} \Big]\\
  &= \E \bigg[\int_0^{t_0} \big\<D U_1(s, Z_s),\sqrt{Q}\, \d W_s\big\> \int_0^{t_0} \big\<Q^{-1/2} B(Z^x_s), \d W_s \big\> \bigg]\\
  &=\E \bigg[\int_0^{t_0} \big\<D U_1(s, Z_s), B(Z^x_s)\big\> \, \d s \bigg] .
  \endaligned
  \end{equation}
By \eqref{lem-1.2} and the Markov property,
  $$ \aligned
  \<D U_1(s, Z_s),h\> &= \E\Big[\phi(e^{tA}x + Z^y_{t-s}) \big\<\Lambda(t-s)h, Q_{t-s}^{-1/2} (Z^y_{t-s}- e^{(t-s)A} y) \big\> \Big]_{y=Z_s} \\
  &= \E\Big[\phi(e^{tA}x + Z_t)\big\< \Lambda(t-s)h, Q_{t-s}^{-1/2} (Z_t- e^{(t-s)A} Z_s) \big\> \big| Z_s \Big] \\
  &= \E\Big[\phi(Z^x_t)\big\< \Lambda(t-s)h, Q_{t-s}^{-1/2} (Z_t- e^{(t-s)A} Z_s)\big\> \big| \mathcal F_s\Big].
  \endaligned $$
Since $Z^x_s$ is $\mathcal F_s$-measurable, we have
  \begin{equation}\label{lem-1.4}
  \big\<D U_1(s, Z_s), B(Z^x_s)\big\>= \E\Big[\phi(Z^x_t)\big\< \Lambda(t-s)B(Z^x_s), Q_{t-s}^{-1/2} (Z_t- e^{(t-s)A} Z_s)\big\> \big| \mathcal F_s\Big].
  \end{equation}
Substituting this formula into the right hand side of \eqref{lem-1.3} yields
  $$\aligned
  \E\Big[ U_1(t_0, Z_{t_0}) M^{(1)}_t \Big] &= \int_0^{t_0} \E\Big[\phi(Z^x_t)\,  \big\<\Lambda(t-s) B(Z^x_s), Q_{t-s}^{-1/2} (Z_t- e^{(t-s)A} Z_s)\big\> \Big] \, \d s.
  \endaligned $$
Finally, we show that we can let $t_0\uparrow t$ on the right hand side. Note that $\phi\in C_b(H)$ and, conditioned on $\mathcal F_s$, $\big\<\Lambda(t-s) B(Z^x_s), Q_{t-s}^{-1/2} (Z_t- e^{(t-s)A} Z_s)\big\>$ is a centered Gaussian random variable with variance $|\Lambda(t-s) B(Z^x_s)|^2$; thus,
  $$\aligned
  &\, \Big| \E\Big[\phi(Z^x_t)\,  \big\<\Lambda(t-s) B(Z^x_s), Q_{t-s}^{-1/2} (Z_t- e^{(t-s)A} Z_s)\big\> \Big] \Big|\\
  \leq &\, \|\phi\|_\infty\, \E\Big[ \E \Big(\big|  \big\<\Lambda(t-s) B(Z^x_s), Q_{t-s}^{-1/2} (Z_t- e^{(t-s)A} Z_s)\big\> \big| \Big| \mathcal F_s \Big) \Big] \\
  \leq &\, \|\phi\|_\infty\, \E \big| \Lambda(t-s) B(Z^x_s) \big| \leq \|\phi\|_\infty \|B\|_\infty \| \Lambda(t-s) \|_{\L(H)}.
  \endaligned $$
Condition (iv) in Hypothesis \ref{hypothe} implies $\int_0^t \| \Lambda(t-s) \|_{\L(H)} \, \d s<\infty$; by the dominated convergence theorem, we can take the limit $t_0\uparrow t$. Taking into account the facts \eqref{lem-1.2.5} and \eqref{lem-1.2.8}, we finish the proof by letting $t_0\uparrow t$.
\end{proof}

\begin{lemma}\label{2-lem-1}
Let $Z_s$ and $U_1(s,y)$ be as in Proposition \ref{lem-1}. It holds that, for any $p\geq 1$,
  $$\lim_{s\uparrow t} \E |U_1(s, Z_s) - U_1(t,Z_t)|^p=0.$$
\end{lemma}

\begin{proof}
Note that $|U_1(s, Z_s)| \leq \|\phi \|_\infty$ $\P$-a.s. for all $s\in [0,t]$; therefore, it suffices to show the limit for $p=1$. In the proofs below we borrow some ideas from \cite[Proposition 6.2]{Cerrai-94}.

\emph{Step 1}. We first show that the family $\mathcal P_T:= \{N_{Q_t}: t\in [0,T]\}$ of Gaussian measures on $H$ is tight. By the Prohorov theorem (see \cite[p. 60, Theorem 5.2]{Billingsley}), it is sufficient to show that $\mathcal P_T$ is weakly compact. Take an arbitrary subset $\{N_{Q_{t_n}} : n\geq 1\} \subset \mathcal P_T$; we have to prove the existence a weakly convergent subsequence. Since $\{t_n : n\geq 1\} \subset[0,T]$, we can find a subsequence $\{t'_n : n\geq 1\}$ which converges to some $t_0\in [0,T]$. By \cite[Proposition 2.3]{DaPrato}, the process $\{Z_t\}_{t\geq 0}$ is continuous in the mean square sense, namely, $\lim_{s\to t} \E |Z_s-Z_t|^2 =0$ for all $t\geq 0$. Then, for any $f\in C_b(H)$,
  $$\lim_{n\to \infty} \int_H f(y)\, N_{Q_{t'_n}}(\d y)= \lim_{n\to \infty} \E f\big(Z_{t'_n} \big) = \E f(Z_{t_0}) = \int_H f(y)\, N_{Q_{t_0}}(\d y), $$
where we have used the dominated convergence theorem in the second equality. This implies that $\big\{ N_{Q_{t'_n}} : n\geq 1 \big\}$ converges weakly to $N_{Q_{t_0}}\in \mathcal P_T$.

\emph{Step 2}. We show that, for any bounded and uniformly continuous $\psi:H\to \R$, it holds
  \begin{equation}\label{2-lem-1.1}
  \lim_{s\to 0} \sup_{x\in K} |S_s\psi(x) - \psi(x)| =0
  \end{equation}
for any compact set $K\subset H$. First, for any $\eps>0$, there is a $\delta >0$ such that
  $$|\psi(x) - \psi(y)| \leq \eps \quad \text{for all } x,y\in H \text{ with } |x-y|\leq \delta. $$
Since the semigroup $\{e^{sA}: s\geq 0\}$ on $H$ is strongly continuous and $K$ is a compact subset of $H$, we can find $s_0>0$ small enough such that
  $$|e^{sA} x-x| \leq \delta/2 \quad \text{for all } s\in (0, s_0],\ x\in K. $$
Now for any $x\in K$ and $s\in (0, s_0]$, we have
  $$\aligned
  |S_s\psi(x) - \psi(x)| &\leq \int_H \big|\psi(e^{sA} x+ y) -\psi(x) \big| \, N_{Q_s}(\d y) \\
  &= \bigg(\int_{\{|y|\leq \delta/2\}} + \int_{\{|y|> \delta/2\}}\bigg) \big|\psi(e^{sA} x+ y) -\psi(x) \big| \, N_{Q_s}(\d y) \\
  &\leq \eps + 2 \|\psi \|_\infty  \int_{\{|y|> \delta/2\}} \, N_{Q_s}(\d y),
  \endaligned $$
where in the third step we have used the fact that $|e^{sA} x+ y -x| \leq \delta$ for all $|y|\leq \delta/2$. Moreover,
  $$\int_{\{|y|> \delta/2\}} \, N_{Q_s}(\d y) \leq \frac4{\delta^2} \int_{H} |y|^2 \, N_{Q_s}(\d y) = \frac4{\delta^2} \text{Tr}(Q_s), $$
which tends to 0 as $s\to 0$. Note that the two estimates above are independent of $x\in K$, thus we obtain \eqref{2-lem-1.1}.

\emph{Step 3}. With the above preparations, we are ready to prove the desired limit. Recall that $t\in (0,T]$ is fixed and, by \eqref{lem-1.1}, $U_1(s,y)= \big( S_{t-s} \tilde\phi \big)(y)$ where $\tilde\phi(y)= \phi(e^{tA} x+y)$ is bounded and uniformly continuous; thus, $U_1(s,Z_s)= \big( S_{t-s} \tilde\phi \big) (Z_s)$ and $U_1(t,Z_t)= \tilde\phi(Z_t)$. As a result,
  \begin{equation}\label{2-lem-1.2}
  \aligned
  \E |U_1(s, Z_s) - U_1(t,Z_t)| &= \E\big|\big( S_{t-s} \tilde\phi \big) (Z_s)- \tilde\phi(Z_t)\big| \\
  &\leq \E\big|\big( S_{t-s} \tilde\phi \big) (Z_s)- \tilde\phi(Z_s)\big| + \E\big|\tilde\phi(Z_s) - \tilde\phi(Z_t)\big|.
  \endaligned
  \end{equation}
We denote the two quantities by $J_1$ and $J_2$ respectively. First, given $\eps >0$, by \emph{Step 1}, we can find a compact set $K_\eps \subset H$ such that $\sup_{s\in [0,T]} N_{Q_s}(H\setminus K_\eps) \leq \eps$; moreover, we deduce from \emph{Step 2} that there is a $\delta_\eps>0$ such that, for all $s\leq \delta_\eps $,
  $$\sup_{y\in K_\eps} \big|\big( S_s \tilde\phi \big) (y)- \tilde\phi(y)\big| \leq \eps. $$
Therefore, for all $s\in [t-\delta_\eps, t]$, we have
  $$\aligned
  J_1 &= \int_H \big|\big( S_{t-s} \tilde\phi \big) (y)- \tilde\phi(y)\big|\, N_{Q_s}(\d y)\\
  &= \bigg(\int_{K_\eps} + \int_{H\setminus K_\eps} \bigg) \big|\big( S_{t-s} \tilde\phi \big) (y)- \tilde\phi(y)\big|\, N_{Q_s}(\d y) \\
  &\leq \eps + 2\big\| \tilde\phi \big\|_\infty N_{Q_s}(H\setminus K_\eps) \leq \eps + 2 \|\phi \|_\infty\, \eps.
  \endaligned $$
Regarding $J_2$, note that $\tilde\phi$ is bounded and uniformly continuous, and $Z_s \to Z_t$ in mean square sense, hence, the dominated convergence theorem implies $\lim_{s\uparrow t} J_2=0$. Combining these results, we finish the proof by letting $s\uparrow t$ in \eqref{2-lem-1.2}.
\end{proof}

Now we are ready to prove the first main result of the paper.

\begin{proof}[Proof of Theorem \ref{main-thm}]
We will prove $I_n(t,x)= v_n(t,x)$ for all $n\geq 1$. Indeed, we will show that ($s_0=t$)
  \begin{equation}\label{proof-2.0}
  \aligned
  I_n(t,x) &= \int_0^t \d s_1 \int_0^{s_1} \d s_2\cdots \int_0^{s_{n-1}} \d s_n \\
  &\quad \E\Bigg[\phi(Z^x_t)\prod_{i=1}^n \big\<\Lambda(s_{i-1}-s_i) B(Z^x_{s_i}), Q_{s_{i-1}-s_i}^{-1/2} (Z_{s_{i-1}}- e^{(s_{i-1}-s_i)A} Z_{s_i}) \big\>  \Bigg].
  \endaligned
  \end{equation}
Once we have this formula, changing the variables $s_i= r_{n+1-i},\ 1\leq i\leq n$ and using \eqref{subsec-proof.1} yield the result.

We prove \eqref{proof-2.0} by induction. Proposition \ref{lem-1} shows the formula holds for $n=1$. Next, assume we have proved \eqref{proof-2.0} for $n-1$, namely, for any bounded and uniformly continuous $\phi:H\to \R$ and $t>0$, it holds that
  \begin{equation}\label{proof-2.0.5}
  \aligned
  I_{n-1}(t,x) &= \E\Big[\phi(Z^x_t) M^{(n-1)}_t \Big] \\
  &= \int_0^t \d s_1 \int_0^{s_1} \d s_2\cdots \int_0^{s_{n-2}} \d s_{n-1} \\
  &\quad \E\Bigg[\phi(Z^x_t)\prod_{i=1}^{n-1} \big\<\Lambda(s_{i-1}-s_i) B(Z^x_{s_i}), Q_{s_{i-1}-s_i}^{-1/2} (Z_{s_{i-1}}- e^{(s_{i-1}-s_i)A} Z_{s_i}) \big\>  \Bigg].
  \endaligned
  \end{equation}
We turn to prove it for $n$. By the definition of $I_n(t,x)$ in \eqref{prop-Girsanov.1}, we have
  $$I_n(t,x)= \E\Big[\phi(Z^x_t) M^{(n)}_t \Big]= \E\Big[U_1(t,Z_t) M^{(n)}_t \Big],$$
where $U_1$ is the function defined in the proof of Lemma \ref{lem-1}. Recall that
  $$M^{(n)}_t= \int_0^t M^{(n-1)}_s \d L_s= \int_0^t M^{(n-1)}_s \big\<Q^{-1/2} B(Z^x_s), \d W_s\big\>,$$
and $\big\{M^{(i)}_t \big\}_{t\geq 0},\, i\geq 1$, are martingales with finite moments of all orders (see the remark above Proposition \ref{prop-Girsanov}). Therefore, similarly as in the proof of Lemma \ref{lem-1}, first applying the It\^o formula to $U_1(s, Z_s)$ on some interval $[0,t_0]$ with $t_0< t$ and then letting $t_0\uparrow t$, we obtain
  \begin{equation}\label{proof-2.1}
  I_n(t,x) =\int_0^t \E\Big[M^{(n-1)}_{s_1} \big\< B(Z^x_{s_1}),D U_1({s_1}, Z_{s_1})\big\> \Big]\, \d s_1.
  \end{equation}

Next, for any $s_1\in [0,t)$, we define $\phi_{s_1}(y):= \big\< B(y),D U_1(s_1, y-e^{s_1A}x) \big\>$ which is bounded and uniformly continuous on $H$; then \eqref{proof-2.1} becomes
  \begin{equation}\label{proof-2.2}
  I_n(t,x) =\int_0^t \E\Big[M^{(n-1)}_{s_1} \phi_{s_1} (Z^x_{s_1}) \Big]\, \d s_1.
  \end{equation}
Applying the induction hypothesis \eqref{proof-2.0.5} to $t=s_1$ and $\phi= \phi_{s_1}$, we have
  $$\aligned
  \E\Big[M^{(n-1)}_{s_1} \phi_{s_1}\big(Z^x_{s_1} \big) \Big] &= \int_0^{s_1} \d s_2 \int_0^{s_2} \d s_3\cdots \int_0^{s_{n-1}} \d s_n \\
  &\quad \E\Bigg[\phi_{s_1}(Z^x_{s_1})\prod_{i=2}^n \big\<\Lambda(s_{i-1}-s_i) B(Z^x_{s_i}), Q_{s_{i-1}-s_i}^{-1/2} (Z_{s_{i-1}}- e^{(s_{i-1}-s_i)A} Z_{s_i}) \big\>  \Bigg].
  \endaligned $$
By the definition of $\phi_{s_1}$,
  $$\aligned
  &\ \E\Bigg[\phi_{s_1}(Z^x_{s_1})\prod_{i=2}^n \big\<\Lambda(s_{i-1}-s_i) B(Z^x_{s_i}), Q_{s_{i-1}-s_i}^{-1/2} (Z_{s_{i-1}}- e^{(s_{i-1}-s_i)A} Z_{s_i}) \big\>  \Bigg] \\
  =&\ \E\Bigg[\big\< B(Z^x_{s_1}),D U_1({s_1}, Z_{s_1})\big\> \prod_{i=2}^n \big\<\Lambda(s_{i-1}-s_i) B(Z^x_{s_i}), Q_{s_{i-1}-s_i}^{-1/2} (Z_{s_{i-1}}- e^{(s_{i-1}-s_i)A} Z_{s_i}) \big\>  \Bigg] \\
  =&\ \E\Bigg[ \E\Big(\phi(Z^x_t) \big\< \Lambda(t-{s_1})B(Z^x_{s_1}), Q_{t-{s_1}}^{-1/2} (Z_t- e^{(t-{s_1})A} Z_{s_1})\big\> \big| \mathcal F_{s_1}\Big) \\
  &\hskip15pt \times \prod_{i=2}^n \big\<\Lambda(s_{i-1}-s_i) B(Z^x_{s_i}), Q_{s_{i-1}-s_i}^{-1/2} (Z_{s_{i-1}}- e^{(s_{i-1}-s_i)A} Z_{s_i}) \big\>  \Bigg],
  \endaligned $$
where in the last step we have used \eqref{lem-1.4} with $s=s_1$. As the second part (the product of $i=2,\cdots,n$) in the expectation is $\mathcal F_{s_1}$-measurable, we arrive at
  $$\aligned
  &\ \E\Bigg[\phi_{s_1}(Z^x_{s_1})\prod_{i=2}^n \big\<\Lambda(s_{i-1}-s_i) B(Z^x_{s_i}), Q_{s_{i-1}-s_i}^{-1/2} (Z_{s_{i-1}}- e^{(s_{i-1}-s_i)A} Z_{s_i}) \big\>  \Bigg] \\
  =&\ \E\Bigg[ \phi(Z^x_t) \prod_{i=1}^n \big\<\Lambda(s_{i-1}-s_i) B(Z^x_{s_i}), Q_{s_{i-1}-s_i}^{-1/2} (Z_{s_{i-1}}- e^{(s_{i-1}-s_i)A} Z_{s_i}) \big\>  \Bigg],
  \endaligned $$
where $s_0=t$. Therefore,
  $$\aligned
  \E\Big[M^{(n-1)}_{s_1} \phi_{s_1}\big(Z^x_{s_1} \big) \Big] &= \int_0^{s_1} \d s_2 \int_0^{s_2} \d s_3\cdots \int_0^{s_{n-1}} \d s_n \\
  &\quad \E\Bigg[ \phi(Z^x_t) \prod_{i=1}^n \big\<\Lambda(s_{i-1}-s_i) B(Z^x_{s_i}), Q_{s_{i-1}-s_i}^{-1/2} (Z_{s_{i-1}}- e^{(s_{i-1}-s_i)A} Z_{s_i}) \big\>  \Bigg].
  \endaligned $$
Inserting this identity into \eqref{proof-2.2} yields the desired formula for $I_n(t,x)$.
\end{proof}

\section{Kolmogorov equations with unbounded nonlinearities} \label{sec-Kolmogorov}

Our purpose here is to prove Theorem \ref{main-thm-0}: the existence part will be proved by using the iteration scheme \eqref{new-iteration}, while the uniqueness part follows by applying the same idea to \eqref{mild-sol}.

First, we describe the idea of proof for the sake of reader's understanding. We shall write in the sequel $\|\cdot \|_{L^p}$ or $ \|\cdot \|_{L^p(\mu)}$ for the norm in $L^{p} (H,\mu),\, p\geq 1$. The same notation will be used for $H$-valued functions. Recall the iteration scheme \eqref{new-iteration}. The growth condition on $B$ implies that $B\in L^q(H,\mu; H)$ for any $q>1$ and, thanks to Corollary \ref{cor-moment} in the appendix, we can obtain explicit estimate on $\|B\|_{L^q(\mu)}$. Using the H\"older inequality, the integrability of $k^n_s$, and thus of $v_{n+1}(s)$, is lower than that of $D v_n(s)$ which has the same integrability as $v_n(s)$. Assume that $v_n(s) \in L^{p_n}(H,\mu)$ for all $s>0$ and $n\geq 1$; then $\{p_n\}_{n\in \N}$ is strictly decreasing. In order to prove Theorem \ref{main-thm-0}, we also need $p_n>\bar p,\, n\in \N$. These considerations lead us to the search of two sequences $\{p_n\}_{n\in \N}$ and $\{q_n\}_{n\in \N}$ such that
  \begin{equation}\label{exponents}
  \frac1{p_n} = \frac1{p_{n-1}} + \frac1{q_n}, \quad n\geq 1.
  \end{equation}
We shall make use of the $L^{q_n}(H,\mu; H)$-norm of the nonlinear function $B:H\to H$. In view of the estimate in Corollary \ref{cor-moment}, the exponents $q_n$ should not grow too fast.

Now we define the two sequences $\{p_n\}_{n\in \N}$ and $\{q_n\}_{n\in \N}$ as follows. Recall the condition on $\beta$ in Hypothesis \ref{hypothe-1}; we can find $\kappa >1$ such that
  \begin{equation}\label{paramets}
  \beta \kappa < 2(1-\delta).
  \end{equation}
Next, for fixed $\bar p\in (1,p_0)$, since $\kappa >1$, there exists $n_0\in \N$ such that
  $$\sum_{n=n_0}^\infty \frac1{n^\kappa} < \frac1{\bar p} - \frac1{p_0}. $$
Set
  \begin{equation}\label{paramets-q}
  q_n= (n+n_0)^\kappa,\quad n\geq 1
  \end{equation}
and determine $p_n$ as in \eqref{exponents}. This implies
  $$\frac1{p_n} = \frac1{p_{0}} + \frac1{(1+n_0)^\kappa} + \frac1{(2+n_0)^\kappa} + \cdots + \frac1{(n+n_0)^\kappa} < \frac1{p_0} + \frac1{\bar p} -\frac1{p_0} = \frac1{\bar p}\, ,$$
thus $p_n> \bar p$ for all $n\in \N$.

Thanks to \eqref{exponents}, if $D v_n(s) \in L^{p_n}(H, \mu; H)$ for all $s>0$, then by H\"older's inequality,
  $$\|k^n_s\|_{L^{p_{n+1}}(\mu)} \leq \|B\|_{L^{q_{n+1}}(\mu)} \|D v_n(s)\|_{L^{p_{n}}(\mu)}. $$
Combining this with the first equality in \eqref{new-iteration}, we can estimate the norms $\|v_{n+1}(s) \|_{L^{p_{n+1}}(\mu)}$ and $\| D v_{n+1}(s) \|_{L^{p_{n+1}}(\mu)}$ (the latter requires the strong Feller property of $\{S_t\}_{t\geq 0}$). According to the above choices  \eqref{exponents}--\eqref{paramets-q} of the parameters, we can show that (see Proposition \ref{prop-convergence}) the two series below are convergent:
  $$\sum_{n=0}^\infty\|v_n(s) \|_{L^{p_{n}}(\mu)} <+\infty, \quad \sum_{n=0}^\infty \|D v_n(s) \|_{L^{p_{n}}(\mu)} <+\infty. $$
Next, since $p_n> \bar p$ for all $n\in \N$, one has
  $$\sum_{n=0}^\infty \|v_n(s) \|_{L^{\bar p}(\mu)} \leq \sum_{n=0}^\infty\|v_n(s) \|_{L^{p_{n}}(\mu)}, \quad \sum_{n=0}^\infty \|D v_n(s) \|_{L^{\bar p}(\mu)} \leq \sum_{n=0}^\infty \|D v_n(s) \|_{L^{p_{n}}(\mu)}. $$
Therefore we conclude that both series
  $$\sum_{n=0}^\infty v_n(s)  \quad \mbox{and}\quad \sum_{n=0}^\infty D v_n(s) $$
converge in $L^{\bar p}(H, \mu)$. The limit is a solution of \eqref{KolE} in $L^{\bar p}(H, \mu)$. The same ideas can be used to prove the uniqueness of solutions, see the proof of Theorem \ref{thm-existence}.

\begin{remark}
\begin{itemize}
\item[\rm(1)] We point out that, in \eqref{paramets-q}, changing the definition as $q_n= (n+n_0)[\log(n+ n_0)]^\kappa\, (\kappa>1)$ will not relax the growth condition on $B$, see {\rm (v$'$)} in Hypothesis \ref{hypothe-1}. This can be easily seen from the proofs of Lemma \ref{lem-estim-moments} and Proposition \ref{prop-convergence} below: we still need $\beta <2(1-\delta)$ to show the following limit
      $$\lim_{n\to \infty} (n+n_0)^{\beta/2} [\log(n+n_0)]^{\beta \kappa/2} \frac{\Gamma(1+(n-1)(1-\delta))} {\Gamma(1+n(1-\delta))}= 0. $$
\item[\rm(2)] In \cite{Cerrai} (see also \cite[Section 2.8.3]{DaPrato}), a weakly continuous semigroup $\{S_t\}_{t\geq 0}$ is defined in the space $UC_m(H)$ of continuous functions with polynomial growth, where $m$ is a positive integer. One might ask whether it is possible to show the existence of solutions to \eqref{KolE} by using the iteration scheme and the growth property of the semigroup. Indeed, in each step of the iteration \eqref{new-iteration}, the growth rates of the functions $v_n(t)$ increase with $n$, due to the multiplication by $B$; moreover, we have to estimate the moments of the form \eqref{lem-moment.1} (cf. \cite[Section 3]{Cerrai}), rather than that in Corollary \ref{cor-moment}. In the end, what we get is a product of a certain factorials of $i$ from 1 to $n$, instead of a single factorial as in Lemma \ref{lem-estim-moments} below. Therefore, unlike Proposition \ref{prop-convergence}, it seems impossible to show that the series obtained is convergent. Note that, in \cite[Section 5]{Cerrai}, the function $B:H \to H$ is assumed to be bounded.
\end{itemize}
\end{remark}

The rest of the section has a similar structure as \cite[Section 2]{FLR}, but we work here in the $L^p$-setting ($p<\infty$). In Section \ref{sec-first-two} we first give some preparations and then provide the formulae and estimates of the first two terms of the iteration process \eqref{new-iteration}. They will give us the clue for the expression and proof of general terms in Section \ref{sec-general}; the convergence of the iteration scheme will also be proved there. The second main result of the paper (Theorem \ref{main-thm-0}) is a consequence of Theorem \ref{thm-existence}.

\subsection{Some preparations and the first two iterations}\label{sec-first-two}

Recall the stochastic convolution $\{ W_A(t)\}_{t\geq 0}$ and the Ornstein-Uhlenbeck process $\{Z_t^x \}_{t\geq 0}$ given at the beginning of the paper; we denote their laws by $N_{Q_t}(\d y)$ and $N_{e^{tA} x, Q_t}(\d y)$, respectively. For any $h\in H$, $\big\<h, Q_t^{-1/2}W_A(t) \big>$ is a centered real Gaussian variable with variance
  \begin{equation}\label{real-Gaussian}
  \E \big\<h, Q_t^{-1/2}W_A(t) \big>^2 =|h|_H^2.
  \end{equation}
We shall write $\mathcal B(H)$ for the space of bounded measurable functions on $H$. The semigroup $\{S_t \}_{t\geq 0}$ associated to $\{Z_t^x \}_{t\geq 0}$ is defined as follows: for any $f\in \mathcal B(H)$ and $t\geq 0$,
  \begin{equation*} %\label{OU-semigroup}
  S_t f(x):= \E f(Z_t^x) = \int_{H} f(y)\, N_{e^{tA} x, Q_t}(\d y) = \int_H f\big( e^{tA} x +y \big)\, N_{Q_t}(\d y).
  \end{equation*}
Recall that the Gaussian measure $\mu$ defined in \eqref{invariant-measure} is the unique invariant measure of $\{S_t \}_{t\geq 0}$. The semigroup $\{S_t \}_{t\geq 0}$ has a unique extension to a strongly continuous semigroup of contractions in $L^p(H,\mu)$, see \cite[Theorem 10.1.5]{DPZ-02}.

It is well known that the semigroup $\{S_t \}_{t\geq 0}$ is strong Feller under conditions (i)--(iii) of Hypothesis \ref{hypothe}. The next result shows its smoothing effect in $L^p(H,\mu)$, see \cite[Proposition 10.3.1]{DPZ-02} for a proof.

\begin{proposition}\label{derivative-formula}
Let $p> 1$. Under (i)--(iii) of the Hypothesis \ref{hypothe}, for any $f\in L^p(H,\mu)$ and $t>0$, we have $S_t f \in W^{1,p}(H,\mu)$ and for any $h\in H$,
  \begin{equation}\label{derivative-formula.1}
  \<h, DS_t f(x)\>= \E \big[ f(Z_t^x) \big\< \Lambda(t)h, Q_t^{-1/2}(Z_t^x- e^{tA}x ) \big\>\big].
  \end{equation}
Moreover,
  \begin{equation}\label{derivative-formula.2}
  \| D S_tf \|_{L^p} \leq C_p \|\Lambda(t)\|_{\mathcal L(H)} \|f\|_{L^p}.
  \end{equation}
\end{proposition}

Indeed, by \cite[Theorem 10.3.5]{DPZ-02}, $S_t f\in C^\infty(H)$ for any $t>0$. But $S_t f$ is in general not bounded. This is easily seen from the example below: if $f(x)= \<x,h\>$ for some $h\in H$, then $S_tf(x)= \<e^{tA}x, h\>, \, x\in H$, which is unbounded.

\subsubsection{The first two terms of the iteration \eqref{new-iteration}}

We begin with the expression and estimates of the first term $v_1(t,x)$.

\begin{proposition}\label{prop-first-term}
For any $t>0$ and $x\in H$,
  \begin{equation}\label{prop-first-term.1}
  v_1(t,x) = \int_0^t \E\Big[\phi(Z_t^x) \big\<\Lambda(s) B(Z_{t-s}^x), Q_s^{-1/2}\big(Z_t^x - e^{sA} Z_{t-s}^x \big) \big\> \Big]\,\d s.
  \end{equation}
Moreover,
  $$\big\|v_1(t) \big\|_{L^{p_1}} \leq \|\phi\|_{L^{p_0}} \|B\|_{L^{q_1}} \int_0^t \|\Lambda(s)\|_{\mathcal L(H)}\,\d s$$
and
  $$\big\|D v_1(t) \big\|_{L^{p_1}} \leq \|\phi\|_{L^{p_0}} \|B\|_{L^{q_1}} \int_0^t \|\Lambda(t-s)\|_{\mathcal L(H)} \|\Lambda(s)\|_{\mathcal L(H)}\,\d s. $$
\end{proposition}

\begin{proof}
The formula \eqref{prop-first-term.1} follows immediately from \eqref{new-iteration} with $n=0$ and the identity below:
  \begin{equation}\label{prop-first-term.1.5}
  \big(S_{t-s} k^0_s \big)(x) = \E\Big[\phi(Z_t^x) \big\<\Lambda(s) B(Z_{t-s}^x), Q_s^{-1/2}\big(Z_t^x - e^{sA} Z_{t-s}^x \big) \big\> \Big].
  \end{equation}
This can be shown by following the proof of \cite[Lemma 3.5]{FLR}. Using the property of conditional expectation,
  $$\aligned
  &\ \E \Big[\phi(Z_t^x) \big\<\Lambda(s) B(Z_{t-s}^x), Q_s^{-1/2} (Z_t^x- e^{sA} Z_{t-s}^x) \big\> \Big]\\
  =&\ \E\Big\{\E \Big[\phi(Z_t^x) \big\<\Lambda(s) B(Z_{t-s}^x), Q_s^{-1/2} (Z_t^x- e^{sA} Z_{t-s}^x) \big\> \big| \mathcal F_{t-s} \Big] \Big\} \\
  =&\ \E\Big\{\E \Big[\phi(Z_t^x) \big\<\Lambda(s) B(Z_{t-s}^x), Q_s^{-1/2} (Z_t^x- e^{sA} Z_{t-s}^x) \big\> \big| Z_{t-s}^x \Big] \Big\},
  \endaligned $$
where the second step is due to the Markov property. Again by the Markov property,
  $$\aligned
  &\ \E \Big[\phi(Z_t^x) \big\<\Lambda(s) B(Z_{t-s}^x), Q_s^{-1/2} (Z_t^x- e^{sA} Z_{t-s}^x) \big\> \big| Z_{t-s}^x \Big]\\
  =&\ \E \Big[\phi(Z_s^y) \big\<\Lambda(s) B(y), Q_s^{-1/2} (Z_s^y- e^{sA} y) \big\> \Big]_{y= Z_{t-s}^x }\\
  =&\ k^0_s(y) \big|_{y= Z_{t-s}^x }= k^0_s(Z_{t-s}^x),
  \endaligned $$
where the second step follows from \eqref{derivative-formula.1}. Substituting this equality into the previous one we obtain the identity \eqref{prop-first-term.1.5}.

Next, by the definition \eqref{new-iteration} of the iteration and H\"older's inequality, for any $s>0$,
  $$ \big\|k^0_s \big\|_{L^{p_1}} \leq \|B\|_{L^{q_1}} \|D v_0(s) \|_{L^{p_0}}, $$
where the parameters satisfy $\frac1{p_1} = \frac1{p_0} +\frac1{q_1}$, see \eqref{exponents}. By Proposition \ref{derivative-formula}, one has
  $$\|D v_0(s) \|_{L^{p_0}}= \|D S_s \phi\|_{L^{p_0}} \leq \|\Lambda(s)\|_{\mathcal L(H)} \|\phi\|_{L^{p_0}} .$$
Therefore,
  \begin{equation}\label{prop-first-term.2}
  \big\|k^0_s \big\|_{L^{p_1}} \leq \|\phi\|_{L^{p_0}} \|B\|_{L^{q_1}} \|\Lambda(s)\|_{\mathcal L(H)}.
  \end{equation}
Now, it is clear that
  $$\aligned
  \|v_1(t) \|_{L^{p_1}} &\leq \int_0^t \big\| S_{t-s} k^0_s\big\|_{L^{p_1}} \,\d s \leq \int_0^t \big\|  k^0_s\big\|_{L^{p_1}} \,\d s \leq \|\phi\|_{L^{p_0}} \|B\|_{L^{q_1}} \int_0^t \|\Lambda(s)\|_{\mathcal L(H)}\,\d s .
  \endaligned $$
Finally, by \eqref{derivative-formula.2},
  $$\aligned
  \|D v_1(t) \|_{L^{p_1}} &\leq \int_0^t \big\|D \big(S_{t-s} k^0_s\big)\big\|_{L^{p_1}} \,\d s \leq \int_0^t \|\Lambda(t-s) \|_{\mathcal L(H)} \big\|k^0_s \big\|_{L^{p_1}}\,\d s ,
  \endaligned $$
which, together with \eqref{prop-first-term.2}, gives us the last estimate.
\end{proof}

Next we turn to the second term $v_2(t,x)$. Similarly, we have

\begin{lemma}\label{lem-sec-term-2}
One has
\begin{align*}
k_{t}^{1}( x) & =\int_0^t \E\Big[ \phi( Z_t^x) \big\<\Lambda(s) B(Z_{t-s}^x), Q_s^{-1/2} (Z_{t}^x- e^{sA} Z_{t-s}^x) \big\> \\
  &\hskip40pt \times \big\<\Lambda(t-s) B(x), Q_{t-s}^{-1/2} (Z_{t-s}^x- e^{(t-s)A}x ) \big\> \Big] \,\d s.
\end{align*}
\end{lemma}

\begin{proof}
The proof is similar to that of \cite[Lemma 2.7]{FLR}. By Proposition \ref{prop-first-term}, for any $t>0$, $v_1(t)\in W^{1,p_1}(H, \mu)$ and
\[\aligned
k_{t}^{1}( x) &=\left\langle B(x), D v_{1} (t,x)  \right\rangle =\, \int_{0}^{t}\big\langle B( x),D \big( S_{t-s} k^0_s\big)(x)  \big\rangle\, \d s.
\endaligned \]%
Recall that \eqref{prop-first-term.2} implies $k^0_s\in L^{p_1}(\mu)$, thus by Proposition \ref{derivative-formula},
  $$k_{t}^{1}( x) =\int_{0}^{t} \E\Big[ k^0_s(Z^x_{t-s}) \big\langle \Lambda(t-s) B(x), Q_{t-s}^{-1/2} \big( Z^x_{t-s} - e^{(t-s) A} x \big) \big\rangle \Big]\, \d s. $$
According to the first part of the proof of Proposition \ref{prop-first-term}, we have
  $$k^0_s(Z^x_{t-s})= \E \Big[\phi(Z_t^x) \big\<\Lambda(s) B(Z_{t-s}^x), Q_s^{-1/2} (Z_t^x- e^{sA} Z_{t-s}^x) \big\> \big| \mathcal F_{t-s} \Big].$$
Note that $\big\langle \Lambda(t-s) B(x), Q_{t-s}^{-1/2} \big( Z^x_{t-s} - e^{(t-s) A} x \big) \big\rangle$ is $\mathcal F_{t-s}$-measurable. Substituting this equality into the one above and using the property of conditional expectation, we obtain the desired result.
\end{proof}

Consequently, we can prove

\begin{proposition}\label{prop-second-iteration}
For any $t>0$ and $x\in H$,
\begin{align*}
  v_{2}(t,x) & =\int_{0}^{t}\! \int_{0}^{s} \E\Big[ \phi( Z_t^x) \big\<\Lambda(r) B(Z_{t-r}^x), Q_r^{-1/2} (Z_t^x - e^{rA} Z_{t-r}^x ) \big\> \\
  &\hskip60pt \times \big\<\Lambda(s-r) B(Z^x_{t-s}), Q_{s-r}^{-1/2} (Z_{t-r}^x- e^{(s-r)A} Z^x_{t-s} ) \big\> \Big] \,\d r\d s.
\end{align*}
Furthermore,
  $$\|v_2(t)\|_{L^{p_2}} \leq \|\phi\|_{L^{p_0}} \|B\|_{L^{q_1}} \|B\|_{L^{q_2}} \int_{0}^{t}\! \int_{0}^{s} \|\Lambda(s-r)\|_{\mathcal L(H)} \|\Lambda(r)\|_{\mathcal L(H)}\,\d r\d s$$
and
  $$\aligned
  \|D v_2(t)\|_{L^{p_2}} \leq \|\phi\|_{L^{p_0}} \|B\|_{L^{q_1}} \|B\|_{L^{q_2}} \int_{0}^{t}\! \int_{0}^{s} \|\Lambda(t-s)\|_{\mathcal L(H)} \|\Lambda(s-r)\|_{\mathcal L(H)} \|\Lambda(r)\|_{\mathcal L(H)}\,\d r\d s.  \endaligned $$
\end{proposition}

\begin{proof}
By Lemma \ref{lem-sec-term-2}, for any $s>0$ and $y\in H$,
\begin{align*}
k_{s}^{1}( y)& =\int_0^s \E\Big[ \phi( Z_s^y) \big\<\Lambda(r) B(Z_{s-r}^y), Q_r^{-1/2} (Z_{s}^y- e^{rA} Z_{s-r}^y ) \big\> \\
  &\hskip40pt \times \big\<\Lambda(s-r) B(y), Q_{s-r}^{-1/2} (Z_{s-r}^y- e^{(s-r)A} y ) \big\> \Big] \,\d r.
\end{align*}
We have%
  $$\aligned
  \E \big[ k^1_s(Z^x_{t-s}) \big] &= \E \bigg\{ \int_0^s \E\Big[ \phi( Z_s^y) \big\<\Lambda(r) B(Z_{s-r}^y), Q_r^{-1/2} (Z_{s}^y- e^{rA} Z_{s-r}^y ) \big\> \\
  &\hskip60pt \times \big\<\Lambda(s-r) B(y), Q_{s-r}^{-1/2} (Z_{s-r}^y- e^{(s-r)A} y ) \big\> \Big]_{y= Z^x_{t-s}} \,\d r \bigg\} \\
  &= \int_0^s \E\Big[ \phi( Z_t^x) \big\<\Lambda(r) B(Z_{t-r}^x), Q_r^{-1/2} (Z_t^x - e^{rA} Z_{t-r}^x ) \big\> \\
  &\hskip45pt \times \big\<\Lambda(s-r) B(Z^x_{t-s}), Q_{s-r}^{-1/2} (Z_{t-r}^x- e^{(s-r)A} Z^x_{t-s} ) \big\> \Big] \,\d r,
  \endaligned $$
where the second step follows from the Markov property. Therefore,
  \begin{align*}
  v_{2} (t,x ) & =\int_{0}^{t} \big( S_{t-s}k_{s}^{1}\big)( x)\,\d s =\int_{0}^{t}\mathbb{E} \big[ k_{s}^{1}( Z_{t-s}^{x}) \big] \d s \\
  & =\int_{0}^{t}\int_{0}^{s} \E\Big[ \phi( Z_t^x) \big\<\Lambda(r) B(Z_{t-r}^x), Q_r^{-1/2} (Z_t^x - e^{rA} Z_{t-r}^x ) \big\> \\
  &\hskip60pt \times \big\<\Lambda(s-r) B(Z^x_{t-s}), Q_{s-r}^{-1/2} (Z_{t-r}^x- e^{(s-r)A} Z^x_{t-s} ) \big\> \Big] \,\d r\d s.
  \end{align*}

Next, by the definition of $k^1_s$ in \eqref{new-iteration} and of the parameters in \eqref{exponents},
  \begin{equation}\label{prop-second-iteration.1}
  \aligned
  \big\| k^1_s \big\|_{L^{p_2}} &\leq \|B\|_{L^{q_2}} \|D v_1(s)\|_{L^{p_1}}\\
  &\leq \|\phi\|_{L^{p_0}} \|B\|_{L^{q_1}} \|B\|_{L^{q_2}} \int_{0}^{s} \|\Lambda(s-r)\|_{\mathcal L(H)} \|\Lambda(r)\|_{\mathcal L(H)}\,\d r,
  \endaligned
  \end{equation}
where in the second step we have used the last inequality in Proposition \ref{prop-first-term}. This implies
  $$\aligned
  \| v_2(t) \|_{L^{p_2}} &\leq \int_0^t \big\| k^1_s \big\|_{L^{p_2}} \,\d s \\
  &\leq \|\phi\|_{L^{p_0}} \|B\|_{L^{q_1}} \|B\|_{L^{q_2}} \int_0^t \int_{0}^{s} \|\Lambda(s-r)\|_{\mathcal L(H)} \|\Lambda(r)\|_{\mathcal L(H)}\,\d r \d s.
  \endaligned $$
Moreover, by Proposition \ref{derivative-formula},
  $$\|D v_2(t) \|_{L^{p_2}} \leq \int_0^t \big\| D \big( S_{t-s}k_{s}^{1}\big) \big\|_{L^{p_2}} \,\d s \leq \int_0^t \|\Lambda(t-s) \|_{\mathcal L(H)} \big\| k^1_s \big\|_{L^{p_2}} \,\d s.$$
Using \eqref{prop-second-iteration.1}, we obtain the last estimate.
\end{proof}

\subsection{The general terms $v_n(t,x)$, estimates and proof of Theorem \ref{main-thm-0}} \label{sec-general}

To proceed further, we rewrite the expression of $v_{2} (t,x )$ as following:
  \begin{align*}
  v_{2} (t,x ) & =\int_{0}^{t} \d s_2 \int_{0}^{s_2} \d s_1\, \E\Big[ \phi( Z_t^x) \big\<\Lambda(s_1) B(Z_{t-s_1}^x), Q_{s_1}^{-1/2} (Z_t^x - e^{s_1 A} Z_{t-s_1}^x ) \big\> \\
  &\hskip80pt \times \big\<\Lambda(s_2-s_1) B(Z^x_{t-s_2}), Q_{s_2-s_1}^{-1/2} (Z_{t-s_1}^x- e^{(s_2-s_1)A} Z^x_{t-s_2} ) \big\> \Big] .
  \end{align*}
Thus, if we denote $s_0=0$, then
  $$\aligned
  v_{2} (t,x )=& \int_{0}^{t} \d s_2 \int_{0}^{s_2} \d s_1\\
  &\ \E\Bigg[ \phi( Z_t^x) \prod_{i=1}^2 \Big\<\Lambda(s_i- s_{i-1}) B(Z_{t-s_i}^x), Q_{s_i-s_{i-1}}^{-1/2} \big( Z_{t-s_{i-1}}^x - e^{(s_i- s_{i-1}) A} Z_{t-s_i}^x \big) \Big\> \Bigg].
  \endaligned $$
This inspires us of the formulae for general terms.

\begin{theorem}\label{thm-estim}
Let $s_0=0$. For any $n\geq 1$,
  \begin{equation}\label{iteration}
  \aligned
  v_{n} (t,x ) =& \int_{0}^{t} \d s_n \int_0^{s_n} \d s_{n-1} \cdots \int_{0}^{s_2} \d s_1\\
  &\, \E\Bigg[ \phi( Z_t^x) \prod_{i=1}^n \Big\<\Lambda(s_i- s_{i-1}) B(Z_{t-s_i}^x), Q_{s_i-s_{i-1}}^{-1/2} \big( Z_{t-s_{i-1}}^x - e^{(s_i- s_{i-1}) A} Z_{t-s_i}^x \big) \Big\> \Bigg].
  \endaligned
  \end{equation}
Moreover,
  $$\|v_n(t)\|_{L^{p_n}} \leq \|\phi\|_{L^{p_0}} \Bigg[ \prod_{i=1}^n \|B\|_{L^{q_i}} \Bigg] \int_{0}^{t} \d s_n \int_0^{s_n} \d s_{n-1} \cdots \int_{0}^{s_2} \d s_1 \, \prod_{i=1}^n \|\Lambda(s_i-s_{i-1})\|_{\mathcal L(H)} $$
and, letting $s_{n+1} =t$,
  $$\|D v_n(t)\|_{L^{p_n}} \leq \|\phi\|_{L^{p_0}} \Bigg[ \prod_{i=1}^n \|B\|_{L^{q_i}} \Bigg] \int_{0}^{t} \d s_n \int_0^{s_n} \d s_{n-1} \cdots \int_{0}^{s_2} \d s_1 \, \prod_{i=1}^{n+1} \|\Lambda(s_i-s_{i-1})\|_{\mathcal L(H)}. $$
\end{theorem}

\begin{proof}
We follow the idea of the proof of \cite[Theorem 2.9]{FLR} and proceed by induction. Indeed, in view of the arguments in Section \ref{sec-first-two}, we shall also prove inductively the formula
  $$\aligned
  k^{n}_t(x)=&\, \int_{0}^{t} \d s_n \int_0^{s_n} \d s_{n-1} \cdots \int_{0}^{s_2} \d s_1 \\
  &\ \E\Bigg[\phi( Z_t^x)\prod_{i=1}^{n+1} \Big\<\Lambda(s_i- s_{i-1}) B(Z_{t-s_i}^x), Q_{s_i-s_{i-1}}^{-1/2} \big( Z_{t-s_{i-1}}^x - e^{(s_i- s_{i-1}) A} Z_{t-s_i}^x \big) \Big\> \Bigg],
  \endaligned $$
where $s_0=0$ and $s_{n+1} =t$. The results in Section \ref{sec-first-two} show that the assertions on $v$ hold for $n=1,\, 2$, and the above formula of $k$ holds with $n=1$. Now we assume the assertions on $v$ (resp. on $k$) hold for $n$ (resp. for $n-1$), and try to prove them in the next iteration.

By the induction hypotheses, we have $v_n(s)\in W^{1,p_n}(H,\mu)$ for all $s>0$ and thus, by the definitions of the iteration \eqref{new-iteration} and of the exponents \eqref{exponents}, $k^n_s \in L^{p_{n+1}}(H, \mu)$ with
  $$\aligned
  \big\| k^n_s \big\|_{L^{p_{n+1}}} &\leq \|B\|_{L^{q_{n+1}}} \|D v_n(s)\|_{L^{p_{n}}}\\
  &\leq \|\phi\|_{L^{p_0}} \Bigg[ \prod_{i=1}^{n+1} \|B\|_{L^{q_i}} \Bigg] \int_{0}^{s} \d s_n \int_0^{s_n} \d s_{n-1} \cdots \int_{0}^{s_2} \d s_1 \, \prod_{i=1}^{n+1} \|\Lambda(s_i-s_{i-1})\|_{\mathcal L(H)},
  \endaligned $$
where $s_{n+1} =s$. Proposition \ref{derivative-formula} implies $S_{t-s} k^n_s\in W^{1,p_{n+1}}(H,\mu)$ for all $s\in (0,t)$, and from the formula
  $$v_{n+1}(t,x) = \int_0^t \big(S_{t-s} k^n_s \big)(x) \,\d s $$
we deduce readily the estimate on $\|v_{n+1}(t)\|_{L^{p_{n+1}}}$. Using \eqref{derivative-formula.2} we can also prove the estimate of $\|D v_{n+1}(t)\|_{L^{p_{n+1}}}$.

Next we prove the formula for $k^n_t(x)$ (note that the induction hypothesis gives us the expression of $k^{n-1}_t(x)$). We have
  \begin{equation}\label{iteration-1}
  \aligned
  k^n_t(x) &= \<B(x), D v_n(t,x)\> = \int_0^t \big\<B(x), D\big( S_{t-s} k^{n-1}_s\big) (x)\big\>\,\d s\\
  &= \int_0^t \E \Big[ k^{n-1}_s(Z^x_{t-s}) \big\< \Lambda(t-s) B(x), Q_{t-s}^{-1/2} (Z^x_{t-s} - e^{(t-s)A} x) \big\> \Big]\,\d s,
  \endaligned
  \end{equation}
where we used Proposition \ref{derivative-formula} in the last step. By the induction hypothesis,
  $$\aligned
  k^{n-1}_s(y)=&\, \int_{0}^{s} \d s_{n-1} \int_0^{s_{n-1}} \d s_{n-2} \cdots \int_{0}^{s_2} \d s_1 \\
  &\ \E\Bigg[\phi( Z_s^y)\prod_{i=1}^{n} \Big\<\Lambda(s_i- s_{i-1}) B(Z_{s-s_i}^y), Q_{s_i-s_{i-1}}^{-1/2} \big( Z_{s-s_{i-1}}^y - e^{(s_i- s_{i-1}) A} Z_{s-s_i}^y \big) \Big\> \Bigg],
  \endaligned $$
where $s_0=0$ and $s_{n} =s$. Therefore, by the Markov property,
  $$\aligned
  &\ k^{n-1}_s(Z^x_{t-s})\\
  =&\, \int_{0}^{s} \d s_{n-1} \int_0^{s_{n-1}} \d s_{n-2} \cdots \int_{0}^{s_2} \d s_1 \\
  &\ \E\Bigg[\phi( Z_t^x)\prod_{i=1}^{n} \Big\<\Lambda(s_i- s_{i-1}) B(Z_{t-s_i}^x), Q_{s_i-s_{i-1}}^{-1/2} \big( Z_{t-s_{i-1}}^x - e^{(s_i- s_{i-1}) A} Z_{t-s_i}^x \big) \Big\> \bigg| \mathcal F_{t-s} \Bigg].
  \endaligned $$
Inserting this identity into \eqref{iteration-1} and noticing that $\big\< \Lambda(t-s) B(x), Q_{t-s}^{-1/2} (Z^x_{t-s} - e^{(t-s)A} x) \big\>$ is measurable with respect to $\mathcal F_{t-s}$, we obtain
  $$\aligned
  k^n_t(x) &= \int_0^t \d s \int_0^s \d s_{n-1} \cdots \int_0^{s_2} \d s_1 \, \E \Bigg\{ \big\< \Lambda(t-s) B(x), Q_{t-s}^{-1/2} (Z^x_{t-s} - e^{(t-s)A} x) \big\> \\
  &\hskip40pt \times \phi( Z_t^x)\prod_{i=1}^{n} \Big\<\Lambda(s_i- s_{i-1}) B(Z_{t-s_i}^x), Q_{s_i-s_{i-1}}^{-1/2} \big( Z_{t-s_{i-1}}^x - e^{(s_i- s_{i-1}) A} Z_{t-s_i}^x \big) \Big\> \Bigg\}.
  \endaligned $$
Renaming $s$ as $s_n$ gives us the formula of $k^n_t(x)$ in the new iteration for all $t>0$ and $x\in H$.

Finally we prove the expression for $v_{n+1}(t,x)$. We have
  $$v_{n+1}(t,x) = \int_0^t \big(S_{t-s} k^n_s \big)(x) \,\d s = \int_0^t \E \big[ k^n_s(Z^x_{t-s}) \big] \,\d s.$$
Using the formula we have just proved for $k^n_s(y)$ and the Markov property, we can obtain the expression for $v_{n+1}(t,x)$ in a similar way as above.
\end{proof}

The formula \eqref{iteration} is suitable for induction arguments in the proof above, due to the convolution structure of the iteration scheme \eqref{new-iteration}. On the other hand, the time parameter of the process $\{Z^x_t \}_{t\geq 0}$ is reversed, and thus \eqref{iteration} is not convenient in numerical computations. By making the change of variables
  $$r_i = t- s_{n+1-i},\quad i=1,\cdots, n,$$
we can obtain the following results.

\begin{corollary}\label{cor-new-estim}
For any $n\geq 1$,
  \begin{equation}\label{iteration-new}
  \aligned
  v_{n} (t,x ) =& \int_{0}^{t} \d r_n \int_0^{r_n} \d r_{n-1} \cdots \int_{0}^{r_2} \d r_1\\
  &\ \E\Bigg[ \phi( Z_t^x) \prod_{i=1}^n \Big\<\Lambda(r_{i+1}- r_{i}) B\big(Z_{r_i}^x \big), Q_{r_{i+1}- r_{i}}^{-1/2} \big( Z_{r_{i+1}}^x - e^{(r_{i+1}- r_{i}) A} Z_{r_i}^x \big) \Big\> \Bigg],
  \endaligned
  \end{equation}
where $r_{n+1} =t$. Furthermore,
  $$\aligned
  \|v_n(t)\|_{L^{p_n}} \leq \|\phi\|_{L^{p_0}} \Bigg[ \prod_{i=1}^n \|B\|_{L^{q_i}} \Bigg] \int_{0}^{t} \d r_n \int_0^{r_n} \d r_{n-1} \cdots \int_{0}^{r_2} \d r_1 \, \prod_{i=1}^n \|\Lambda(r_{i+1}-r_i)\|_{\mathcal L(H)}
  \endaligned$$
and, setting $r_0=0$,
  $$\|D v_n(t)\|_{L^{p_n}} \leq \|\phi\|_{L^{p_0}} \Bigg[ \prod_{i=1}^n \|B\|_{L^{q_i}} \Bigg] \int_{0}^{t} \d r_n \int_0^{r_n} \d r_{n-1} \cdots \int_{0}^{r_2} \d r_1 \, \prod_{i=0}^n \|\Lambda(r_{i+1}-r_i)\|_{\mathcal L(H)}.$$
\end{corollary}

Next, recall the definition of $q_i$ in \eqref{paramets-q}; using the moment estimates of Gaussian measures in Corollary \ref{cor-moment}, we can prove

\begin{lemma}\label{lem-estim-moments}
For any $n\in \N$,
  $$\prod_{i=1}^n \|B\|_{L^{q_i}} \leq C_\beta^n ({\rm Tr}\, Q_\infty)^{n\beta /2} \big[(n+n_0)!\big]^{\beta \kappa/2}. $$
\end{lemma}

\begin{proof}
According to the condition (v$'$) in Hypothesis \ref{hypothe-1},
  $$\|B\|_{L^{q_i}} \leq C \Big(1+ \big\| |\cdot|^\beta \big\|_{L^{q_i}} \Big) \leq C \Big(1+ C_\beta ({\rm Tr}\, Q_\infty)^{\beta/2} q_i^{\beta/2} \Big) \leq \tilde C_\beta ({\rm Tr}\, Q_\infty)^{\beta/2} q_i^{\beta/2}, $$
where the second inequality follows from Corollary \ref{cor-moment} below. Recalling that $q_i = (i+n_0)^\kappa$, we obtain the desired result.
\end{proof}

The next technical result is proved in \cite[Lemma 3.12]{FLR}.

\begin{lemma}\label{lem-multiple-integral}
Assume $\delta\in (0,1)$ and $n\geq 1$. Let $r_0=0$ and $r_{n+1}=t$. One has
  $$\int_{0}^{t} \d r_n \int_0^{r_n} \d r_{n-1} \cdots \int_{0}^{r_2} \d r_1 \prod_{i=1}^n \frac1{(r_{i+1} -r_{i})^\delta } = \frac{\Gamma(1-\delta)^n} {\Gamma(1+n(1-\delta))} t^{n(1-\delta)} $$
and
  $$\int_{0}^{t} \d r_n \int_0^{r_n} \d r_{n-1} \cdots \int_{0}^{r_2} \d r_1 \prod_{i=0}^n \frac1{(r_{i+1} -r_{i})^\delta } = \frac{\Gamma(1-\delta)^{n+1}} {\Gamma((n+1) (1-\delta))} t^{n(1-\delta)- \delta} .$$
Here $\Gamma(\cdot)$ is the Gamma function.
\end{lemma}

\begin{proof}
We include the proof here for the reader's convenience. First we prove
  \begin{equation}\label{beta-equality}
  \int_{0}^{t} \d r_n \int_0^{r_n} \d r_{n-1} \cdots \int_{0}^{r_2} \d r_1 \prod_{i=1}^n \frac1{(r_{i+1} -r_i)^\delta } = t^{n(1-\delta)} \prod_{i=1}^n B\big(1-\delta, 1+(i-1)(1-\delta) \big),
  \end{equation}
where $B(\alpha, \beta)$ is the Beta function:
  $$B(\alpha, \beta) = \int_0^1 \theta^{\alpha -1} (1-\theta)^{\beta -1}\,\d\theta, \quad \alpha, \beta>0. $$
We proceed by induction. For $n=1$, noting that $r_2=t$, we change the variable $\theta= r_1/t$ and get
  $$\int_0^t \frac{\d r_1}{(t -r_1)^\delta} = t^{1-\delta} \int_0^1 \frac{\d \theta}{(1 -\theta)^\delta} = t^{1-\delta} \int_0^1 \theta^0 (1-\theta)^{-\delta} \,\d \theta = t^{1-\delta} B(1-\delta, 1). $$
Therefore the equality holds when $n=1$. Now suppose the equality holds for $n-1$, we prove it for $n$. By the induction hypothesis,
  $$\int_0^{r_n} \d r_{n-1} \cdots \int_{0}^{r_2} \d r_1 \prod_{i=1}^{n-1} \frac1{(r_{i+1} -r_i)^\delta} = r_n^{(n-1)(1-\delta)} \prod_{i=1}^{n-1} B\big(1-\delta, 1+(i-1)(1-\delta) \big),$$
thus, noticing that $r_{n+1} =t$,
  $$\int_{0}^{t} \d r_n \int_0^{r_n} \d r_{n-1} \cdots \int_{0}^{r_2} \d r_1 \prod_{i=1}^n \frac1{(r_{i+1} -r_i)^\delta} = \prod_{i=1}^{n-1} B\big(1-\delta, 1+(i-1)(1-\delta) \big) \int_0^t \frac{r_n^{(n-1)(1-\delta)}}{(t-r_n)^\delta}\,\d r_n .$$
We have, by changing variable $\theta = r_n/t$,
  $$\int_0^t \frac{r_n^{(n-1)(1-\delta)}}{(t-r_n)^\delta}\,\d r_n = t^{n(1-\delta)} \int_0^1 \theta^{(n-1)(1-\delta)} (1-\theta)^{-\delta} \,\d\theta = t^{n(1-\delta)} B\big(1-\delta, 1+(n-1)(1-\delta) \big) . $$
Substituting this result into the previous one gives us the identity \eqref{beta-equality}.

Next, it is well known that
  $$B(\alpha, \beta) = \frac{\Gamma(\alpha) \Gamma(\beta)}{\Gamma(\alpha+\beta)}.$$
Therefore,
  $$\prod_{i=1}^n B\big(1-\delta, 1+(i-1)(1-\delta) \big) =  \prod_{i=1}^n \frac{\Gamma(1-\delta) \Gamma(1+(i-1)(1-\delta))} {\Gamma(1+i(1-\delta))} = \frac{\Gamma(1-\delta)^n}{\Gamma(1+n(1-\delta)} . $$
Combining this with \eqref{beta-equality} we obtain the desired formula.

The proof of the second identity is similar, by first establishing the identity
  $$\int_{0}^{t} \d r_n \int_0^{r_n} \d r_{n-1} \cdots \int_{0}^{r_2} \d r_1 \prod_{i=0}^n \frac1{(r_{i+1} -r_i)^\delta } = t^{n(1-\delta) -\delta} \prod_{i=1}^n B\big(1-\delta, i(1-\delta) \big).$$
We omit the details here.
\end{proof}

The estimates below can be deduced easily from Corollary \ref{cor-new-estim} and Lemmas \ref{lem-estim-moments} and \ref{lem-multiple-integral}.

\begin{corollary}\label{cor-estim}
For any $n\in \N$ and $t>0$,
  $$\aligned
  \|v_n(t)\|_{L^{p_n}} \leq \|\phi\|_{L^{p_0}} C_\delta^n C_\beta^n ({\rm Tr}\, Q_\infty)^{n\beta /2} \big[(n+n_0)!\big]^{\beta \kappa/2} \frac{\Gamma(1-\delta)^n} {\Gamma(1+n(1-\delta))} t^{n(1-\delta)}
  \endaligned$$
and, setting $r_0=0$,
  $$\|D v_n(t)\|_{L^{p_n}} \leq \|\phi\|_{L^{p_0}} C_\delta^n C_\beta^n ({\rm Tr}\, Q_\infty)^{n\beta /2} \big[(n+n_0)!\big]^{\beta \kappa/2} \frac{\Gamma(1-\delta)^{n+1}} {\Gamma((n+1) (1-\delta))} t^{n(1-\delta)- \delta}.$$
\end{corollary}

Now we can prove

\begin{proposition}\label{prop-convergence}
Fix any $T>0$. Then the series
  $$\sum_{n=0}^\infty v_n(t,x) \quad \mbox{and} \quad  t^\delta \sum_{n=0}^\infty D v_n(t,x)$$
converge in $C \big([0,T], L^{\bar p}(H,\mu) \big) $.
\end{proposition}

\begin{proof}
We only prove the convergence of the first series, since the other one can be done similarly in view of the second estimate in Corollary \ref{cor-estim}. According to the first estimate in Corollary \ref{cor-estim} and the ratio test, we need to show
  $$\lim_{n\to \infty} (n+n_0)^{\beta \kappa/2} \frac{\Gamma(1+(n-1)(1-\delta))} {\Gamma(1+n(1-\delta))}= 0. $$
Since $\beta \kappa/2 < 1 -\delta $ by \eqref{paramets}, it is sufficient to show that
  \begin{equation}\label{prop-convergence.1}
  \frac{\Gamma(1+(n-1)(1-\delta))} {\Gamma(1+n(1-\delta))} \lesssim \frac1{n^{1-\delta}} \quad \mbox{as } n \to \infty.
  \end{equation}

For simplicity of notation we set $\alpha = 1-\delta$. Using the fact that $\Gamma(1+t)= t\Gamma(t)$ for any $t>0$, we have
  $$\frac{\Gamma(1+ (n-1)\alpha)} {\Gamma(1+ n\alpha)} = \frac{(n-1)\alpha}{n\alpha} \cdot \frac{(n-1)\alpha -1}{n\alpha -1} \cdots \frac{\alpha_1}{\alpha_1 +\alpha} \cdot \frac{\Gamma(\alpha_1)}{\Gamma(\alpha_1 +\alpha)}, $$
where $\alpha_1:= (n-1)\alpha- \lfloor (n-1)\alpha \rfloor +1$ and $\lfloor (n-1)\alpha \rfloor$ is the integer part of $(n-1)\alpha$. The reason that we consider $\alpha_1$ instead of $(n-1)\alpha- \lfloor (n-1)\alpha \rfloor$ is because the latter might be 0, while $\Gamma(0)$ is not well defined. For any $i\leq \lfloor (n-1)\alpha \rfloor -1$,
  $$\frac{(n-1)\alpha -i}{n\alpha -i}= 1- \frac{\alpha}{n\alpha -i} ,$$
and using the simple inequality $\log(1+t) <t$ for all $t>-1$, we have
  $$\log\frac{(n-1)\alpha -i}{n\alpha -i}<  - \frac{\alpha}{n\alpha -i},\quad i=0,1,\cdots, \lfloor (n-1)\alpha \rfloor -1 .$$
Therefore,
  $$\log \frac{\Gamma(1+ (n-1)\alpha)} {\Gamma(1+ n\alpha)} < -\alpha\bigg(\frac1{n\alpha} + \frac1{n\alpha-1} + \cdots + \frac1{\alpha_1 +\alpha} \bigg) + \log \frac{\Gamma(\alpha_1)}{\Gamma(\alpha_1 +\alpha)}.$$
Moreover,
  $$\aligned
  \frac1{n\alpha} + \frac1{n\alpha-1} + \cdots + \frac1{\alpha_1 +\alpha} &\geq \frac1{\lfloor n\alpha \rfloor +1} + \frac1{\lfloor n\alpha \rfloor} + \cdots + \frac1{\lfloor \alpha_1 +\alpha \rfloor +1} \\
  &\geq \log(\lfloor n\alpha \rfloor +1) + \gamma - C \geq \log(n\alpha) + C_\alpha,
  \endaligned $$
where $\gamma\approx 0.57$ is the Euler constant and $C>0$ is some constant depending on $\alpha$. Hence,
  $$\log \frac{\Gamma(1+ (n-1)\alpha)} {\Gamma(1+ n\alpha)} \leq -\alpha \big(\log(n\alpha) + C_\alpha \big) + \log \frac{\Gamma( \alpha_1)}{\Gamma( \alpha_1 +\alpha)},$$
which implies
  $$\frac{\Gamma(1+ (n-1)\alpha)} {\Gamma(1+ n\alpha)} \leq \frac{e^{-\alpha C_\alpha}}{(n\alpha)^\alpha} \frac{\Gamma( \alpha_1)}{\Gamma( \alpha_1 +\alpha)} = : \frac{\tilde C_\alpha}{n^\alpha}.$$
Recalling that $\alpha = 1-\delta$, thus \eqref{prop-convergence.1} holds and the proof is complete.
\end{proof}

Finally, the main result (Theorem \ref{main-thm-0}) follows from the next theorem.

\begin{theorem}\label{thm-existence}
The limit defined in Proposition \ref{prop-convergence} is a mild solution to the Kolmogorov equation \eqref{KolE}. Moreover, it is unique among solutions with the following property: for any $\bar p\in (1,p_0)$, $u\in C\big([0,T], L^{\bar p}(H, \mu) \big)$ and $t^\delta Du(t) \in C\big([0,T], L^{\bar p}(H, \mu;H) \big)$.
\end{theorem}

\begin{proof}
\emph{Step 1: Existence}. According to the iteration scheme, we have
  \begin{equation}\label{thm-existence.1}
  \aligned
  u_n(t,x) &= S_t \phi(x) + \int_0^t S_{t-s}\big( \<B, D u_{n-1}(s)\> \big)(x)\,\d s \\
  &= S_t \phi(x) + \sum_{i=0}^{n-1} \int_0^t S_{t-s}\big( \<B, D v_{i}(s)\> \big)(x)\,\d s .
  \endaligned
  \end{equation}
The left hand side converges in $C \big([0,T], L^{\bar p}(H,\mu) \big)$ to the limit $u(t,x)$. Next, we show the sum on the right hand side is also convergent. Indeed, for any $n<m$,
  $$\aligned
  \sum_{i=n}^{m} \int_0^T \big\| S_{t-s}\big( \<B, D v_{i}(s)\> \big)\big\|_{L^{\bar p}}\,\d s & \leq \sum_{i=n}^{m} \int_0^T \big\| \<B, D v_{i}(s)\> \big\|_{L^{\bar p}}\,\d s \\
  &\leq \sum_{i=n}^{m} \int_0^T \big\| \<B, D v_{i}(s)\> \big\|_{L^{p_{i+1}}}\,\d s \\
  &\leq \sum_{i=n}^{m} \|B\|_{L^{q_{i+1}}} \int_0^T \big\| D v_{i}(s) \big\|_{L^{p_i}}\,\d s,
  \endaligned $$
where in the last two steps we have used the fact that $\bar p <p_{i+1}$ and H\"older's inequality with the exponents $\frac1{p_{i+1}} = \frac1{p_i} + \frac1{q_{i+1}}$, see \eqref{exponents}. By Corollary \ref{cor-estim}, we have
  $$\aligned
  &\ \sum_{i=n}^{m} \int_0^T \big\| S_{t-s}\big( \<B, D v_{i}(s)\> \big)\big\|_{L^{\bar p}}\,\d s \\
  \leq &\ \sum_{i=n}^{m} \|B\|_{L^{q_{i+1}}} \|\phi\|_{L^{p_0}} C^i ({\rm Tr}\, Q_\infty)^{i\beta /2} \big[(i+n_0+1)!\big]^{\beta \kappa/2} \frac{\Gamma(1-\delta)^{i+1}} {\Gamma((i+1) (1-\delta))} \int_0^T s^{i(1-\delta)- \delta} \,\d s \\
  \leq &\ \|\phi\|_{L^{p_0}} \sum_{i=n}^{m} \|B\|_{L^{q_{i+1}}} C^i ({\rm Tr}\, Q_\infty)^{i\beta /2} \big[(i+n_0+1)!\big]^{\beta \kappa/2} \frac{\Gamma(1-\delta)^{i+1}} {\Gamma((i+1) (1-\delta))} \frac{T^{(i+1)(1-\delta)}}{(i+1)(1-\delta)},
  \endaligned  $$
where $C= C_{\delta, \beta}>0$. Recalling the proof of Lemma \ref{lem-estim-moments}, we have
  $$\frac{\|B\|_{L^{q_{i+1}}(\mu)}}{i+1} \leq \tilde C_\beta ({\rm Tr}\, Q_\infty)^{\beta/2} \frac{q_{i+1}^{\beta/2}}{i+1} \leq \tilde C_\beta ({\rm Tr}\, Q_\infty)^{\beta/2} \frac{(i+n_0+1)^{\beta\kappa /2}}{i+1},$$
which, due to $\beta\kappa /2<1$, tends to 0 as $i\to \infty$. Therefore, by the proof of Proposition \ref{prop-convergence}, the above sum tends to 0 as $m>n \to \infty$. Finally we let $n\to \infty$ in \eqref{thm-existence.1} to conclude that the limit $u(t,x)$ solves the mild formulation of the Kolmogorov equation \eqref{KolE}.

\emph{Step 2: Uniqueness}. Let $u_1$ and $u_2$ be two solutions with the above-mentioned properties. Fix any $\tilde p, \bar p \in (1,p_0)$ with $\tilde p< \bar p$. Then we have $u_i\in C\big([0,T], L^{\bar p}(H, \mu) \big)$ and $t^\delta Du_i(t) \in C\big([0,T], L^{\bar p}(H, \mu;H) \big)$, $i=1,2$.

Similarly as the discussions at the beginning of this section, we can define two sequences $\{\bar p_n\}_{n\in \N}$ and $\{\bar q_n\}_{n\in \N}$ with the properties below:
\begin{itemize}
\item $\bar p_0 =\bar p$;
\item $\frac1{\bar p_n} = \frac1{\bar p_{n-1}} + \frac1{\bar q_n}$ for all $ n\in \N $;
\item $\bar p_n >\tilde p$ for all $n\in \N$.
\end{itemize}
Indeed, it suffices to define
  $$\bar q_n = (n+ n_1)^\kappa, \quad n\in \N,$$
where $n_1\in \N$ verifies
  $$\sum_{n= n_1}^\infty \frac1{n^\kappa} < \frac1{\tilde p} - \frac1{\bar p} .$$

Since $u_1$ and $u_2$ solve the mild formulation \eqref{mild-sol}, we have
  \begin{equation}\label{eq-uniqueness}
  u_1(t) - u_2(t) = \int_0^t S_{t-s}\big(\<B, D(u_1(s) - u_2(s))\>\big)\,\d s.
  \end{equation}
Therefore,
  $$D(u_1(t) - u_2(t))= \int_0^t D S_{t-s}\big(\<B, D(u_1(s) - u_2(s))\>\big)\,\d s.$$
Proposition \ref{derivative-formula} implies, for any $n>1$,
  $$\aligned
  \|D(u_1(t) - u_2(t))\|_{L^{\bar p_n}} &\leq \int_0^t \big\| D S_{t-s}\big(\<B, D(u_1(s) - u_2(s))\>\big) \big\|_{L^{\bar p_n}} \,\d s \\
  & \leq \int_0^t \|\Lambda(t-s)\|_{\L(H)} \big\| \<B, D(u_1(s) - u_2(s))\> \big\|_{L^{\bar p_n}} \,\d s.
  \endaligned $$
By H\"older's inequality,
  $$\|D(u_1(t) - u_2(t))\|_{L^{\bar p_n}} \leq \int_0^t \|\Lambda(t-s)\|_{\L(H)} \|B\|_{L^{\bar q_n}} \| D(u_1(s) - u_2(s))\|_{L^{\bar p_{n-1}}} \,\d s. $$
Repeating the above procedure, we obtain
  $$\aligned
  \|D(u_1(t) - u_2(t))\|_{L^{\bar p_n}}\leq \Bigg[ \prod_{i=1}^n \|B\|_{L^{\bar q_i}} \Bigg] & \int_0^t \d s_n \int_0^{s_n} \d s_{n-1} \cdots \int_0^{s_2} \d s_1  \\
  & \Bigg[ \prod_{i=1}^n \|\Lambda(s_{i+1} -s_i)\|_{\L(H)} \Bigg] \| D(u_1(s_1) - u_2(s_1))\|_{L^{\bar p_{0}}} ,
  \endaligned $$
where $s_{n+1} =t$. According to the assumption on $s^\delta D u_i(s)\ (i=1,2)$, there exists a constant $C>0$ such that
  $$\| D(u_1(s) - u_2(s))\|_{L^{\bar p_{0}}} \leq C/ s^\delta, \quad s>0.$$
Combining these facts with Hypothesis \ref{hypothe-1}-(iv$'$), we get
  $$\|D(u_1(t) - u_2(t))\|_{L^{\bar p_n}}\leq C C_\delta^n \Bigg[ \prod_{i=1}^n \|B\|_{L^{\bar q_i}} \Bigg] \int_0^t \d s_n \int_0^{s_n} \d s_{n-1} \cdots \int_0^{s_2} \d s_1 \, \prod_{i=0}^n \frac1{(s_{i+1} -s_i)^\delta}, $$
where $s_0=0$. Now applying Lemma \ref{lem-estim-moments} yields (replacing $q_i$ by $\bar q_i$)
  $$\prod_{i=1}^n \|B\|_{L^{\bar q_i}} \leq C_\beta^n ({\rm Tr}\, Q_\infty)^{n\beta /2} \big[(n+n_1)!\big]^{\beta \kappa/2}. $$
This together with Lemma \ref{lem-multiple-integral} leads to
  $$\|D(u_1(t) - u_2(t))\|_{L^{\bar p_n}}\leq C^n ({\rm Tr}\, Q_\infty)^{n\beta /2} \big[(n+n_1)!\big]^{\beta \kappa/2} \frac{\Gamma(1-\delta)^{n+1}} {\Gamma((n+1) (1-\delta))} t^{n(1-\delta)- \delta}, $$
where $C=C_{\delta, \beta}$. Since $\bar p_n> \tilde p$ for all $n\in \N$, we finally obtain
  $$\|D(u_1(t) - u_2(t))\|_{L^{\tilde p}}\leq C^n ({\rm Tr}\, Q_\infty)^{n\beta /2} \big[(n+n_1)!\big]^{\beta \kappa/2} \frac{\Gamma(1-\delta)^{n+1}} {\Gamma((n+1) (1-\delta))} t^{n(1-\delta)- \delta}.$$
The proof of Proposition \ref{prop-convergence} shows that, for any $t>0$, the right hand side vanishes as $n\to \infty$. Therefore, $Du_1(t)= Du_2(t)$ for all $t\in (0, T]$. Taking into account \eqref{eq-uniqueness}, we conclude the uniqueness of solutions.
\end{proof}

\section{Appendix: some moment estimates on Gaussian measures}

We have the following estimate on the moments of $\mu= N_{Q_\infty} = N_{0, Q_\infty}$.

\begin{lemma}\label{lem-moment}
For any $n\in \N$,
  \begin{equation}\label{lem-moment.1}
  \int_{H} |x|^{2n} \,\d\mu(x) \leq 2^n (n!) ({\rm Tr}\, Q_\infty)^{n}.
  \end{equation}
\end{lemma}

\begin{proof}
To save notation we write $Q$ instead of $Q_\infty$ in the proof below. We use the following fact (see \cite[Proposition 1.13]{DaPrato-1}):
  $$F(s):= \int_H e^{s|x|^2}\,\d\mu(x)= \big[{\rm det}(1-2s Q) \big]^{-1/2}, \quad s< 1/(2\lambda_1), $$
where $\lambda_1>0$ is the biggest eigenvalue of $Q$. It is clear that
  \begin{equation}\label{lem-moment.1.5}
  F^{(n)}(0)= \int_{H} |x|^{2n} \,\d\mu(x).
  \end{equation}
One has the useful identity (see \cite[Example 1.15]{DaPrato-1}):
  $$F'(s)= F(s) {\rm Tr}\big(Q(1-2s Q)^{-1}\big). $$
Then by the combinatorial formula,
  $$F^{(n+1)}(s) = \frac{\d^n}{\d s^n} \big[F(s) {\rm Tr}\big(Q(1-2s Q)^{-1}\big) \big] = \sum_{k=0}^n \binom n k F^{(n-k)}(s) \frac{\d^k}{\d s^k} {\rm Tr}\big(Q(1-2s Q)^{-1}\big),$$
where $\binom n k$ are the combinatorial numbers. Regarding $Q$ as a diagonal matrix and using induction, it is easy to show that
  $$\frac{\d^k}{\d s^k} {\rm Tr}\big(Q(1-2s Q)^{-1}\big) = 2^k (k!) {\rm Tr}\big(Q^{k+1}(1-2s Q)^{-(k+1)}\big). $$
Therefore,
  $$F^{(n+1)}(0) = \sum_{k=0}^n 2^k (k!) \binom n k F^{(n-k)}(0) {\rm Tr}\big(Q^{k+1} \big).$$
Using the very rough inequality
  \begin{equation}\label{ineq-trace}
  {\rm Tr}\big(Q^{k+1} \big) \leq ({\rm Tr}\, Q)^{k+1},
  \end{equation}
we obtain
  \begin{equation}\label{lem-moment.2}
  F^{(n+1)}(0) \leq \sum_{k=0}^n 2^k (k!) \binom n k F^{(n-k)}(0) ({\rm Tr}\, Q)^{k+1}.
  \end{equation}

Now we prove the assertion by induction. Assume that the estimate \eqref{lem-moment.1} holds for $k\leq n$; we want to prove it for $k=n+1$. Taking into account \eqref{lem-moment.1.5}, this follows immediately from \eqref{lem-moment.2}. Indeed, by the induction hypothesis,
  $$\aligned
  F^{(n+1)}(0) &\leq \sum_{k=0}^n 2^k (k!) \binom n k 2^{n-k} ((n-k)!) ({\rm Tr}\, Q)^{n-k} ({\rm Tr}\, Q)^{k+1} \\
  &= 2^n ({\rm Tr}\, Q)^{n+1} \sum_{k=0}^n (n!) \leq 2^{n+1} ((n+1)!) ({\rm Tr}\, Q)^{n+1} .
  \endaligned $$
The proof is complete.
\end{proof}

\begin{remark}
The estimate \eqref{ineq-trace} looks very rough, but in a sense it is also sharp. For instance, assume that $Q$ is a diagonal matrix and that the entries on the diagonal are decreasing. Let $Q_{22} \downarrow 0$ (thus $Q_{ii} \downarrow 0$ for all $i>2$), then ${\rm Tr}\big(Q^{k+1} \big) \to Q_{11}^{k+1}$ and $({\rm Tr}\, Q)^{k+1} \to Q_{11}^{k+1}$.
\end{remark}

\begin{corollary}\label{cor-moment}
There exists a constant $C>0$ such that for any $p>1$,
  $$\bigg(\int_H |x|^p \,\d\mu(x)\bigg)^{1/p} \leq C({\rm Tr}\, Q_\infty)^{1/2} \sqrt{p}.$$
\end{corollary}

\begin{proof}
By Lemma \ref{lem-moment} and the Stirling formula, for any $n\in \N$,
  $$\bigg( \int_{H} |x|^{2n} \,\d\mu(x) \bigg)^{1/(2n)} \leq (2\, {\rm Tr}\, Q_\infty)^{1/2} (n!)^{1/(2n)} \leq C ( {\rm Tr}\, Q_\infty )^{1/2} \sqrt{n}\, . $$
As a result,
  $$\bigg( \int_{H} |x|^{n} \,\d\mu(x) \bigg)^{1/n} \leq C ( {\rm Tr}\, Q_\infty )^{1/2} \sqrt{n}\, . $$
Hence, for any $p>1$, denoting by $\lceil p \rceil$ the smallest integer which is greater than $p$,
  $$\bigg(\int_H |x|^p \,\d\mu(x)\bigg)^{1/p} \leq \bigg(\int_H |x|^{\lceil p \rceil} \,\d\mu(x)\bigg)^{1/\lceil p \rceil} \leq C ( {\rm Tr}\, Q_\infty )^{1/2} \sqrt{\lceil p \rceil} $$
which implies the desired estimate with the constant $\sqrt{2}\, C$.
\end{proof}

\end{document}